\documentclass{article}
\usepackage{amsmath, amssymb, amsthm, hyperref, color,} 
\usepackage[body={152mm, 215mm}, a4paper]{geometry}
\usepackage[all]{xy}

\makeatletter
\def\url@leostyle{%
  \@ifundefined{selectfont}{\def\UrlFont{\sf}}{\def\UrlFont{\small\ttfamily}}}
\makeatother
\theoremstyle{plain}
\newtheorem{thm}{Theorem}[section]
\newtheorem{lem}[thm]{Lemma}
\newtheorem{cor}[thm]{Corollary}
\newtheorem{prop}[thm]{Proposition}

\theoremstyle{definition}

\newtheorem{rmq}[thm]{Remark}
\newtheorem{exm}[thm]{Example}

\input cyracc.def \font\tencyr=wncyr10 \def\russe{\tencyr\cyracc} 
\def\Sha{\text{\russe{Sh}}}

\DeclareMathOperator{\Spec}{Spec}
\DeclareMathOperator{\pic}{Pic}
\DeclareMathOperator{\Sel}{Sel}
\DeclareMathOperator{\Cl}{Cl}
\DeclareMathOperator{\homr}{Hom}
\DeclareMathOperator{\coker}{coker}
\DeclareMathOperator{\diviseur}{div}

\DeclareMathOperator{\Divp}{DivPrinc}
\DeclareMathOperator{\Div}{Div}
\DeclareMathOperator{\logpic}{LogPic}
\DeclareMathOperator{\img}{Im}
\DeclareMathOperator{\Br}{Br}
\DeclareMathOperator{\Gal}{Gal}

\newcommand{\Selphi}{\Sel^{\varphi}(E/K)}
\newcommand{\Selphihat}{\Sel^{\hat\varphi}(E'/K)}

\newcommand{\A}{\mathcal{A}}
\newcommand{\E}{\mathcal{E}}
\newcommand{\OK}{\mathcal{O}_K}

\newcommand{\FF}{\mathbb{F}}
\newcommand{\ZZ}{\mathbb{Z}}
\newcommand{\QQ}{\mathbb{Q}}
\newcommand{\Zp}{{}^{\ZZ}\!/\!{}_{p\ZZ}}
\newcommand{\QZ}{{}^{\QQ}\!/\!{}_{\ZZ}}
\newcommand{\cyclic}[1]{{}^{\ZZ}\!/\!{}_{#1 \ZZ}}

\renewcommand{\leq}{\leqslant}
\renewcommand{\geq}{\geqslant}

\begin{document}

\title{The class group pairing and
$p$-descent on elliptic curves}

\author{Jean Gillibert \and Christian Wuthrich\footnote{The second author was supported by the EPSRC grant EP/G022003/1}}

\date{October 2011}

\maketitle

\begin{abstract}
We give explicit formulae for the logarithmic class group pairing on an elliptic curve defined over a number field. Then we relate it to the descent relative to a suitable cyclic isogeny. This allows us to connect the resulting Selmer group with the logarithmic class group of the base. These constructions are explicit and suitable for computer experimentation. From a conceptual point of view, the questions that arise here are analogues of ``visibility'' questions in the sense of Cremona and Mazur.
\end{abstract}


\section{Introduction}
The class-invariant homomorphism has been introduced by Martin Taylor~\cite{t88} in order to study the Galois module structure of extensions obtained when dividing points by isogenies between abelian varieties. For instance, let $p$ be an odd prime and let $\varphi\colon E \to E'$ be an isogeny of degree $p$ between elliptic curves over a number field $K$. Suppose that the kernel of $\varphi$ is isomorphic to $\mu_p$ as a Galois module. For $Q\in E'(K)$ consider the extension $K(\varphi^{-1}(Q))/K$ obtained by adjoining the coordinates of the preimages of $Q$ under $\varphi$. It is a Kummer extension, and we may ask for the arithmetic properties of this extension. It is easy to prove that this extension defines a $\mu_p$-torsor $\varphi^{-1}(Q)$ over the ring $\mathcal{O}_{K,S}$ of $S$-integers of $K$ where $S$ is the set of all bad places for $E$ and of all places above $p$. We have an exact sequence
\begin{equation*}
 \xymatrix@1{0  \ar[r] & \mathcal{O}_{K,S}^{\times}/p \ar[r] & H_{\rm fl}^1(\mathcal{O}_{K,S},\mu_p) \ar[r] & \Cl(\mathcal{O}_{K,S})[p] \ar[r] & 0 }
\end{equation*}
where the right-hand side map can be interpreted as measuring the Galois module structure of $\mu_p$-torsors. The class-invariant homomorphism $\psi^{\rm cl}:E'(K)/\varphi(E(K))\rightarrow \Cl(\mathcal{O}_{K,S})[p]$ is obtained by sending a point $Q\in E'(K)$ to the Galois module structure of the torsor $\varphi^{-1}(Q)$. It has been conjectured by Taylor, and proved by Taylor et al. that torsion points belong to the kernel of $\psi^{\rm cl}$ (see~\cite{pa98}). In other words, if $Q$ is torsion, then $K(\varphi^{-1}(Q))=K(\sqrt[p]{\alpha})$ for some $\alpha\in\mathcal{O}_{K,S}^{\times}$.
However, the behaviour of $\psi^{\rm cl}$ on points of infinite order is not fully understood.

It has been proved by Agboola~\cite{a94} that $\psi^{\rm cl}$ has an alternative ``geometric'' description. Namely, if we let $P\in E'(K)$ be the generator of the kernel of the dual isogeny $\hat\varphi$, then
$$
\psi^{\rm cl}(Q)=\langle P,Q\rangle^{\rm cl}_S
$$
where $\langle \cdot, \cdot\rangle^{\rm cl}_S$ is the so-called class group pairing introduced by Mazur and Tate~\cite{mt}. Therefore our Galois structure problem can now be reformulated in terms of the intersection theory on the N\'eron model of $E'$ over $\mathcal{O}_{K,S}$.

In order to study the Galois module structure of $\varphi^{-1}(Q)$ over the full ring of integers, one needs to handle bad places. The first author has been investigating this. Let $\E$ (resp. $\E'$) be the N\'eron model of $E$ (resp. $E'$) over $X =\Spec(\mathcal{O}_K)$, and let $\E^0(X)$ (resp. $\E'^0(X)$) be the subgroup of points that have everywhere good reduction. Again, we have a class-invariant homomorphism
$$
\psi^0\colon \E'^0(X)/\varphi(\E^0(X))\longrightarrow \Cl(K)[p]
$$
that can also be described using the class group pairing. It is proved in~\cite{gil1} that torsion points belong to the kernel of $\psi^0$, but it is not possible in general to extend this morphism $\psi^0$ to all of $E'(K)/\varphi(E(K))$. To overcome this, the first author introduced a more general pairing, called the logarithmic class group pairing, that has values in the so-called logarithmic class group $\logpic(X,S)$. This pairing combines the class group pairing by Mazur and Tate with Grothendieck's monodromy pairing (see~\cite{gil5}). As a result, we get a class-invariant homomorphism
$$
\psi\colon E'(K)/\varphi(E(K))\longrightarrow \logpic(X,S)[p]
$$
which generalises both $\psi^{\rm cl}$ and $\psi^0$. Of course, this new morphism is no longer expected to vanish on torsion points, because the monodromy pairing is known to be non degenerate.

In this paper, we give explicit formulae for the logarithmic class group pairing on elliptic curves, which are suited for implementation in a computer algebra system. (See section \ref{explicit_pairing_subsec}, in particular Lemma~\ref{explicit_pairing_lem}, Lemma~\ref{e_zero_lem} and the sections~\ref{mult_case_subsubsec} and~\ref{additive_case_subsubsec}.) This allows us to conduct numerical experimentations on these objects that, up to now, have been considered from a purely conceptual point of view. 
Our derivation of the formulae leads us to study the intersection property of certain divisors on the minimal regular model of $E$ as cubic logarithmic line bundles, similar to what was done in~\cite{bosch_lorenzini} for the monodromy pairing.

Under some additional hypotheses (see below in section~\ref{hyp_subsec}) on $E$ and $\varphi$, we perform a descent via the isogeny $\varphi$. We explicitly describe the Selmer group $\Selphi$ as a subgroup of the fppf cohomology group $H_{\rm fl}^1(U_1,\mu_p)$ where $U_1$ is obtained by removing from $X$ a certain subset $S_1$ of the set of places where $E$ has split multiplicative reduction (see section~\ref{descent_sec} for details). At another small set $S_2$ of places of  $K$, we have to impose local conditions. Theorem~\ref{pdescent_thm} shows that we have an exact sequence
\begin{equation*}
 \xymatrix@1{0\ar[r] & \Selphi \ar[r] & H_{\rm fl}^1(U_1,\mu_p) \ar[r] & \bigoplus_{v \in S_2} H_{\rm fl}^1(\mathcal{O}_v , \mu_p) } 
\end{equation*}
and Theorem~\ref{global_duality_thm} describes the cokernel on the right in terms of the Selmer group $\Selphihat$ for the dual isogeny. The methods for the descent are very much inspired by~\cite{ms} except that we use fppf cohomology.

Our interest then focuses on the link between the logarithmic class-invariant homomorphism and the Selmer group.
Using Kummer theory in log flat topology, we get an exact sequence
$$
\xymatrix@1{
0  \ar[r] & \mathcal{O}_K^{\times}/p \ar[r] & H_{\rm fl}^1(U_1,\mu_p)  \ar[r] & \logpic(X,S_1)[p] \ar[r] & 0. }
$$
Let $\psi_{\Sel}\colon\Selphi\to\logpic(X,S_1)[p]$ be the restriction of the map on the right to the subgroup $\Selphi$. 
As our main ingredient to the following investigation, we prove in Theorem~\ref{pairing_and_selmer_thm} that the class-invariant homomorphism $\psi$ is equal to the  composition of the Kummer map and the map $\psi_{\Sel}$, in other words the following diagram commutes:
 \begin{equation*}
  \xymatrix@1@C+3mm{0\ar[r] & E'(K)/\varphi(E(K)) \ar[r]^(0.55){\kappa}\ar[rd]_(0.5){\psi} & \Selphi \ar[d]^{\psi_{\Sel}} \ar[r]  & \Sha(E/K)[\varphi] \ar[r] &  0 \\
  & & \logpic(X,S_1)[p] }
  \end{equation*}
where $\Sha(E/K)$ is the Tate-Shafarevich group of $E/K$, and the horizontal line is exact. This links the logarithmic class group pairing to the descent data.

By the kernel-cokernel sequence, we get an exact sequence
$$\xymatrix@1{
0\ar[r] & \ker(\psi) \ar[r] & \ker(\psi_{\Sel}) \ar[r] & \Sha(E/K)[\varphi]  \ar[r] & \coker(\psi) \ar[r] & \coker(\psi_{\Sel}) \ar[r] & 0}
$$
Understanding this sequence can be viewed in terms of the notion of ``visibility" in the sense of Cremona and Mazur~\cite{cm00}: Instead of comparing the descent maps for two isogenies, we compare $\varphi$-descent with Kummer theory in log flat topology. 
For instance when $\coker(\psi)=0$, one could say, by analogy with the terminology in~\cite{cm00}, that all elements in $\Sha(E/K)[\varphi]$ are \emph{explained} by units.

In the special case when the set $S_2$ is empty, then $\ker(\psi_{\Sel})=\mathcal{O}_K^{\times}/p$ and $\coker(\psi_{\Sel})=0$. Hence $\Sha(E/K)[\varphi]$ is completely determined by the kernel and the cokernel of $\psi$. In particular, this happens if $E$ has everywhere good reduction, in which case these relations have been pointed out previously (see~\cite{ct95} for a survey).

The behaviour of $\psi$ on points of infinite order has to be explored. In Theorem~\ref{Levinlike_thm} we prove that, given a curve $E'$ over $\mathbb{Q}$ with a $p$-torsion point $P\in E'(\mathbb{Q})$, there exists an infinity of imaginary quadratic fields $K$ and a point $Q\in E'(K)$ of infinite order such that $\langle P,Q\rangle^{\rm cl}_S\neq 0$. We also remark that $\psi$ is injective when $K$ is a quadratic imaginary field. Finally, we compute some examples of points of infinite order defined over real quadratic fields that lie or do not lie in the kernel of $\psi$.

Let us review briefly the contents of this paper. At the end of the introduction, we give general hypotheses and notations. In section~\ref{local_sec} and~\ref{descent_sec}, we perform the descent. In section~\ref{class_group_pairing_sec}, which is independent of the rest of the paper, we give explicit formulae for the logarithmic class group pairing on an elliptic curve. In section~\ref{link_sec}, we describe the connection between Selmer groups and logarithmic class groups via the pairing. In section~\ref{quadratic_sec}, we focus on quadratic fields, and in final section~\ref{numerical_sec} we give examples over fields of higher degree.

\subsection{Hypotheses}\label{hyp_subsec}

Let $K$ be a number field and $E/K$ an elliptic curve. Throughout this paper, except for section~\ref{class_group_pairing_sec}, we will assume the following hypotheses.
\begin{enumerate}\renewcommand{\theenumi}{(Hyp~\arabic{enumi})}
  \item $p$ is an odd prime.                                                             \label{odd_asp}
  \item $E$ admits an isogeny $\varphi\colon E \to E'$ of degree $p$ defined over $K$. 
  \item The kernel of the dual isogeny $\hat\varphi\colon E' \to E$ is generated by a $p$-torsion point $P\in E'(K)[p]$. \label{genmup_asp}
  \item If $p=3$, none of the bad fibres of $E$ nor of $E'$ is of type IV or IV${}^*$.                \label{comp3_asp} 
\renewcommand{\theenumi}{(Hyp~\arabic{enumi}')}
   \item $K/\QQ$ is not ramified at $p$.\label{ramp_hyp}
\end{enumerate}

The last condition~\ref{ramp_hyp} can be weakened quite a bit once some additional notation is introduced. All conclusions in this article hold with condition~\ref{ramp_hyp} replaced by hypothesis~\ref{atp_hyp} in section~\ref{atp_subsec}. For instance for semistable curves we will only need that the ramification $e_v$ index at all places $v$ above $p$ satisfies $e_v < p-1$.

While the first three hypotheses are just the setting in which we have placed ourselves, the fourth is merely to exclude that $p$ might divide the order of the group of components at additive places.

\subsection{Notations}

If $v$ is a place of the number field $K$, we write $K_v$, $\mathcal{O}_v$, $\mathfrak{m}_v$, and $k_v$ for the completion at $v$, its ring of integers, its maximal ideal and its residue field respectively. 

We will write $\E$ for the N\'eron model of $E$ over the ring of integers $\OK$. Let $\E^0$ be the connected component of the identity and let $\Phi$ be the groups of components. 
So we obtain an exact sequence of sheaves on the fppf site
\begin{equation}\label{componentdef_eq}
 \xymatrix@1{0  \ar[r] & \E^0 \ar[r] & \E \ar[r] & \Phi \ar[r] & 0. }
\end{equation}
For all finite places $v$ of $K$, we denote by $\Phi_v$ the group of components of the fibre of $\E$ at $v$. Of course, the sheaf $\Phi$ can be decomposed as
$$
\Phi=\bigoplus_{v} \Phi_v
$$
where $v$ runs through the set of places of bad reduction of $\E$.
The number of components of the fibre of $\E$ at $v$, i.e. the rank of the $\Phi_v$, will be denoted by $m_v$, while the number of $k_v$-rational points $\# \Phi_v(k_v)$, the so-called Tamagawa number, is written $c_v$. 
We fix a choice of a  N\'eron differential, denoted by $\omega$. Similarly $\E'$, $\E'^0$, $\Phi'$, $m'_v$, $c'_v$ and $\omega'$ denote the corresponding objects for the elliptic curve $E'$.
 By abuse of notation, we also denote by $\varphi\colon \E\to \E'$ the unique morphism between the N\'eron models extending the isogeny $\varphi$. Note that $\varphi$ induces a map $\varphi\colon \E^0 \to \E'^0$ and another on the components $\varphi\colon \Phi \to \Phi'$.

For any ring $R$, we will denote the quotient of the unit group $R^{\times}$ by its $p$-th powers, simply by $R^{\times}/p$.

All cohomology groups are computed with respect to the fppf site. We will write $H^1(\OK, \cdot)$ as a short hand for $H^1(\Spec(\OK),\cdot)$, similarly for $H^1(\mathcal{O}_v,\cdot)$.

\section{Local information on the isogeny}\label{local_sec}

In this section $v$ will be a finite place of $K$. (Since $p$ is odd, we never have to worry about infinite places.) We wish to describe the isogeny $\varphi$ locally at the level of N\'eron models.

If $v$ is not above $p$, it follows immediately that $\E'[\hat\varphi]$ is isomorphic to $\Zp$ as a group scheme over the ring of integers $\mathcal{O}_v$ of $K_v$ and hence, by Cartier duality, $\E[\varphi]$ is  $\mu_p$. This is not necessarily the case for places above $p$ as we will see later.

\begin{lem}
  For any $v\nmid p$, the curves $E$ and $E'$ have semistable reduction at $v$. 
\end{lem}
\begin{proof}
  Suppose the reduction is additive at $v$.
  There is a point $P$ of order $p$ on the curve $E'$ defined over $K_v$. By assumption~\ref{odd_asp} and~\ref{comp3_asp}, the group of components $\Phi'$ of $\E'$ has order prime to $p$, so $P$ reduces to a point in $\E'^0$. Since the reduction map is injective on points of order $p$ as $v \nmid p$, the reduced point has order $p$, too. But the additive group has no $p$-torsion over $k_v$, so we reached a contradiction.
\end{proof}

We will first describe the cases of multiplicative reduction, both for $v$ above $p$ and away from $p$.
There are four cases to consider, depending on whether the reduction is split or non-split multiplicative and depending on which of the two isogenous curves has more connected components. We choose to denominate the latter property by saying the place is ``forward'' or ``backward'' for $\varphi$.

\subsection{The forward split case} 
In this case, the reduction is \textbf{split} multiplicative and $\E$ has \textbf{less} components than $\E'$.

The elliptic curves are isomorphic to Tate curves over $K_v$ for certain parameters $q$ and $q'$, respectively. We have $E(K_v) = K_v^{\times}/q^{\ZZ}$ and $E'(K_v) = K_v^{\times}/(q')^{\ZZ}$. Since the number of components is given by the valuation of the parameters, we find that $q' = q^p$ and that the isogeny $\varphi\colon E(K_v) \to E'(K_v)$ is induced by the map $x\mapsto x^p$ on $K_v^{\times}$. We conclude that $\Phi'/\varphi(\Phi)=\Zp$ and the cokernel $E'(K_v)/\varphi(E(K_v))$ is isomorphic to $K_v^{\times}/p$. Also we see that $\E[\varphi]$ is contained in $\E^0$.

The dual map $\hat\varphi_{K_v}$ is then induced by $x\mapsto x$. It follows that $\E'[\hat\varphi] = \Phi'[\hat\varphi] = \Zp$, in other words the point $P$ reduces to the singular point at $v$. However, $\hat\varphi\colon E'(K_v)\to E(K_v)$ and hence $\hat\varphi\colon \Phi' \to \Phi$ are surjective. Furthermore, we also see that $\varphi^*(\omega') = p \omega$ and $\hat\varphi^*(\omega) = \omega'$ for the N\'eron differentials.

\subsection{The backward split case}\label{S2_subsec} 
Here the reduction is still \textbf{split} multiplicative, but $\E$ has \textbf{more} components than $\E'$.

The situation is as above except that $E$ and $E'$ switch their places. This time, we have $q = (q')^p$ and the map $\varphi_{K_v}$ identifies with $x\mapsto x$ and its dual with $x\mapsto x^p$. So
$\E^0[\varphi]$ is trivial and $\E[\varphi] = \Phi[\varphi]$. Since $\Phi$ is a constant group scheme, we find that $\mu_p = \Zp$ over $\mathcal{O}_v$, which implies that $\# k_v \equiv 1 \pmod{p}$. In particular, this case can not occur for a place above $p$. This time $\varphi_{K_v}$ is surjective onto $E'(K_v)$, but the cokernel of $\hat\varphi_{K_v}$ is isomorphic to $K_v^{\times}/p$.

For the dual isogeny, we get $\E'[\hat\varphi] \subset \E'^0$ and so $\Phi'[\hat\varphi] = 0$ and $\Phi/\hat\varphi(\Phi') = \Zp$.

\subsection{The backward non-split case} 
Now, the reduction is \textbf{non-split} multiplicative and $\E$ has \textbf{more} components than $\E'$.

Let $L_w/K_v$ be the quadratic unramified extension. Over $L_w$ the reduction becomes split, yet the N\'eron model stays the same. In particular, over $L_w$ we are in the backward split case.

We start by proving that this case can not occur when $v$ is above $p$. The Tamagawa number $c'_v$ is coprime to $p$ by assumption, so the point $P\in E'(K_v)$ must lie on the identity component. Since the group of non-singular point on the reduction has order coprime to $p$, the point $P$ must lie in the formal group. 
Since $\hat\varphi$ has a non-trivial kernel in the formal group we must have that $\hat\varphi^*(\omega) = \omega'$. This will still be the same over the quadratic extension $L_w$ since the N\'eron model does not change. But this is in contradiction with what we found in the backward split case. Hence $v\nmid p$.
 
Let $E^{\dagger}/K_v$ be the twist of $E$ corresponding to $L_w/K_v$. Then $E^{\dagger}$ admits an isogeny $\varphi^{\dagger}$ of degree $p$, for which the place is also backward split.
So $\E[\varphi] = \mu_p$ becomes $\Zp$ over $\mathcal{O}_w$, but they are not isomorphic yet over $\mathcal{O}_v$. In fact $\E[\varphi]$ is the twist of $\E^{\dagger}[\varphi^{\dagger}]$, which is isomorphic to $\Zp$. Hence $\E[\varphi]=\mu_p$ is a quadratic twist of $\Zp$. This implies that $\#k$ must be congruent to $-1$ modulo $p$.

We obtain all information about $\varphi$ by comparing the corresponding information for $E^{\dagger}$ and $E\times L_w$. For instance, 
$\E[\varphi] = \Phi[\varphi]$ and $\Phi'/\varphi(\Phi) = 0$. Similarly $\E'[\hat\varphi] \subset \E'^0$ and $\Psi[\hat\varphi] = 0$ and $\Phi/\hat\varphi(\Phi') = \mu_p$.

The cokernel of $\hat\varphi\colon E(K_v)\to E'(K_v)$ is the quotient of $L_w^{\times}/p$ by $K_v^{\times}/p$. This is equal to $\Zp$, because $L_w$ contains $\mu_p$, but $K_v$ does not. The cokernel of $\varphi_{K_v}$ is trivial. 

\subsection{The forward non-split case}
Finally, we have \textbf{non-split} multiplicative reduction, but $\E$ has \textbf{less} components than $\E'$.

This case is impossible in fact. Since $\Phi$ is of order strictly smaller than $\Phi'$, the induced map $\hat\varphi\colon \Phi'\to \Phi$ can not be injective. Hence the point $P$ can not reduce to the connected component. But this is impossible because $c'_v = \# \Phi'(k_v)$ is coprime to~$p$ if the reduction is non-split multiplicative.

\subsection{Summary}\label{summary_subseq}
Here is a table that summarises all results so far in the section.
\begin{center}
  \begin{tabular}{l|c|c|c}
   Case                    & forward  & backward     & backward \\ \hline &&& \\[-2ex]
   Reduction               & split    & split       & non-split \\
   Components              & $pm_v=m'_v$  & $m_v = p m'_v$  &  $m_v=p m'_v$ \\
   Tamagawas               & $ c'_v=pc_v$ & $c_v = p c'_v$  & $c_v=c'_v=1\text{ or }2$\\ \hline &&& \\[-2ex]
  $\E^0[\varphi]$          & $\mu_p$  & 0           & 0 \\
 $\Phi[\varphi]$           &  0       & $\mu_p=\Zp$ & $\mu_p$\\
 $\Phi'/\varphi(\Phi)$     &  $\Zp$   & 0           & 0\\
  $E'(K_v)/\varphi$        & $K_v^{\times}/p$ & 0 & 0 \\ \hline &&& \\[-2ex]
 $\E'^0[\hat\varphi]$      &  0       & $\Zp$       &$\Zp$\\
 $\Phi'[\hat\varphi] $     &  $\Zp$   & 0           & 0\\
 $\Phi/\hat\varphi(\Phi')$ &  0       & $\Zp$       & $\mu_p$\\
  $E(K_v)/\hat\varphi$     &  0       & $K_v^{\times}/p$ & $\Zp$\\ \hline &&& \\[-2ex]
 $\# k_v $                 &  any     & $\equiv 1 \pmod{p} $   & $\equiv -1 \pmod{p}$     
  \end{tabular}
\end{center}

\subsection{Places above $p$}\label{atp_subsec}
Let $v$ be a place above $p$. We have seen so far that the reduction at $v$ is either good, forward split multiplicative or additive. Since the Tamagawa numbers at additive places are coprime to $p$, the point $P\in E'(K_v)[\hat\varphi]$ will reduce to a singular point only in the forward split case. 

Recall that we fixed N\'eron differentials $\omega$ and $\omega'$ for $E$ and $E'$ respectively. If $z$ is such that $\varphi^*(\omega') = z \cdot \omega$ then we call its valuation $a_v(\varphi) = v(z)$ the \textit{N\'eron scaling} of $\varphi$ at $v$. Similar, we define $a_v(\hat\varphi)$ and we have $a_v(\varphi) + a_v(\hat\varphi) = e_v = v(p)$. It can be computed easily from the power series expansion of the isogeny restricted to the formal group; in fact, it is simply the valuation of the leading term.

We can now make precise how we can weaken the condition~\ref{ramp_hyp} in the introduction. A place $v$ is called \textit{forward} if $a_v(\hat\varphi) = 0$ and \textit{backward} if $a_v(\varphi) = 0 $.
Note that forward split places above $p$ are indeed forward, so there is no contradition with the previous definition.
\begin{enumerate}
\renewcommand{\theenumi}{(Hyp~\arabic{enumi})}\setcounter{enumi}{4}
  \item All places above $p$ are either forward or backward for $\varphi$. Moreover, if a place $v\mid p$ is backward, then we impose that $\mu_p(K_v) = 1$. \label{atp_hyp}
\end{enumerate}
Of course, if $p$ is unramified in $K/\QQ$ then this condition is automatically satisfied, in other words hypothesis~\ref{atp_hyp} implies hypothesis~\ref{ramp_hyp}.
We will see that this new hypothesis may only restrict the situation in the presence of places of good supersingular reduction above $p$ or additive places with potential supersingular reduction. The former can only arise when the ramification index at the place is larger than $p-1$.

\begin{lem}
 If the reduction is good ordinary or potentially good ordinary at a place $v$ above $p$ 
 then either $a_v(\varphi) = 0$ or $a_v(\hat\varphi) = 0$.
\end{lem}
\begin{proof}
  We treat first the case when $E$ has good ordinary reduction. Since the formal group is of height $1$, either $\varphi$ or $\hat\varphi$ induces on the reduction a separable isogeny. A separable isogeny between elliptic curves induces an injection on differential forms, so the N\'eron scaling for this isogeny has to be zero.

  More precisely, if $P$ reduces to a non-zero point in the reduction, then $\hat\varphi$ is separable on the reduction. Hence $a_v(\hat\varphi) = 0$ and so the place is forward. 

  Otherwise, if $P$ reduces to the zero point in the reduction, then $\hat\varphi$ induces a purely inseparable isogeny on the reduction, so $\varphi$ has to induce a separable isogeny. Hence $a_v(\varphi) =0$ and the place is backward. In this case, $P$ belongs to the formal group and by~\cite[Theorem IV.6.1]{sil1} this implies that $e_v > p-1$.

  The cases of multiplicative reduction have already been treated, so we may pass to the case when $E$ has additive, potentially good or multiplicative reduction at $v$. So the Kodaira type is I${}_n^*$ for some $n\geq 0$ and there is a quadratic ramified extension $L_w/K_v$ such that $E\times L_w$ has semistable reduction. Write $w$ for its normalised valuation. Fix a minimal equation of $E/K_v$. Let $u$ be coefficient needed in the transformation of the Weierstrass equation to become minimal over $L_w$: More precisely $u$ will be such that $u\cdot \omega_{E/L_w} =  \omega_{E/K_v}$ holds for the corresponding N\'eron differentials. Therefore the minimal discriminants satisfy $\Delta_{E/K_v} = u^{12}\cdot \Delta_{E/L_w}$. From~\cite[p. 365]{sil2} we find that $v(\Delta_{E/K_v} ) = 6+n$. The type over $L_w$ will be I${}_{2n}$ and so $w(\Delta_{E/L_w}) = 2n$. Hence we get $2\cdot (6+n) = 12\cdot w(u) + 2n$ and we conclude that $u$ is a uniformiser of $L_w
 $. The same argument will apply to the corresponding coefficient $u'$ for $E'$. Write $z_w$ for the constant such that $\varphi^*(\omega_{E'/L_w}) = z_w\cdot \omega_{E/L_w}$ whose valuation is the N\'eron scaling of $\varphi$ over $L_w$. Now
  $$\varphi^*(\omega_{E'/K_v} ) = \varphi^*( u' \cdot \omega_{E'/L_w} ) = u' \cdot z_w \cdot \omega_{E/L_w}  = \frac{u'}{u} \cdot z_w \cdot \omega_{E/K_v}$$
  shows that if the N\'eron scaling of $\varphi$ over $L_w$ is zero, then $a_v(\varphi) =0$. By the same argument, if the N\'eron scaling of $\hat\varphi$ over $L_w$ is zero, then $a_v(\hat\varphi)=0$. So the previous treated semistable ordinary cases imply the result for the additive potentially ordinary case.
\end{proof}
If instead the reduction is good and supersingular, then both $a_v(\varphi)$ and $a_v(\hat\varphi)$ are positive as $[p]$ is purely inseparable. It follows as before that this can only occur when $e_v > p-1$.  For instance, if $E'$ is the curve 37a1 and $K$ is the field $\QQ(P)$ over which it admits a point $P$ of order $3$, then $a_v(\varphi) = 6$ and $a_v(\hat\varphi) = 2$ for one of the places above $3$.
At the supersingular places above $p$ the special fibre of the group schemes $\E[\varphi]$ and $\E'[\hat\varphi]$ are both equal to the group scheme $\alpha_p$.

\begin{lem}
  Let $v$ be a place above $p$. 
  If $e_v < p -1$ and $E$ is semistable, then $v$ is forward.
\end{lem}
\begin{proof}
  We have seen in the proof of the previous Lemma that all good places are forward if $e_v < p-1$. The same is true for split multiplicative places and there are no non-split multiplicative places above $p$.
\end{proof}
Note that the semistability condition is necessary: The example of the elliptic curve labeled 126a1 in Cremona's tables shows that there are curves even defined over $\QQ$ with additive backwards places.

Lemma~3.8 in~\cite{sch} says that for an isogeny $\psi\colon A \to A'$ between elliptic curves, we have
\begin{equation*}
  \# A'(K_v)/\psi \bigl(A(K_v)\bigr) = (\# k_v)^{a_v(\psi)}\cdot \# A(K_v)[\psi] \cdot \frac{c_v(A')}{c_v(A)}
\end{equation*}
where the last is the ratio of Tamagawa numbers. Let $n_v$ be the degree of $K_v$ over $\QQ_p$. Taking into account that $\#E(K_v)[\varphi] = \#\mu_p(K_v) = 1$ for all backward places above $p$, we deduce from it the following.
\begin{equation}\label{coker_phi_at_p_eq}
  \# E'(K_v) / \varphi(E(K_v)) = 
   \begin{cases}
     p^{n_v+1} & \text{if $v$ is forward split multiplicative}\\
     p^{n_v} & \text{otherwise} \\
     1       & \text{if $v$ is backward}
   \end{cases}
\end{equation}
and similarly that
\begin{equation}\label{coker_phihat_at_p_eq}
  \# E(K_v) / \hat\varphi(E'(K_v)) = 
   \begin{cases}
     1 & \text{if $v$ is backward}\\
     p & \text{otherwise}\\
     p^{n_v+1} &\text{if $v$ is forward split multiplicative}
   \end{cases}
\end{equation}
Note that forward places above $p$ where $E$ does not have multiplicative reduction appear in the middle terms labeled ``otherwise''.

\begin{lem}\label{constant_lem}
 Let $v$ be a place above $p$ which is forward. Then $\E'[\hat\varphi]=\Zp$ as (finite flat) group schemes over $\mathcal{O}_v$. If moreover the reduction is not additive, then $\E[\varphi] = \mu_p$.
\end{lem}
\begin{proof}
 If the reduction is multiplicative, we know  the result already. Assume that the reduction is not multiplicative, so $P$ reduces to a non-singular point.

 Let $t$ be a parameter of the formal group of $E'$ with respect to a minimal Weierstrass model.
 We must have that the formal expansion of $\hat\varphi$ is of the form $t$ times a unit power series in $t$ with coefficients in $\mathcal{O}_v$. It induces therefore an isomorphism on the formal groups. Hence the point $P$ in the kernel of $\hat\varphi$ can not lie in the formal group and so it reduces to a non-zero, non-singular point in the reduction, i.e. it meets the special fibre of $\E'$ at a non-zero point of the connected component. Hence $\E'[\hat\varphi]$ is indeed isomorphic to $\Zp$ over $\mathcal{O}_v$.

 If the reduction is good, then the kernel of $\varphi$ is the Cartier dual of the kernel of $\hat\varphi$ and hence it is equal to $\mu_p$.
\end{proof}

We note that we can not say that $\E[\varphi]=\mu_p$ if the reduction is additive. In fact the kernel may well not be quasi-finite and hence not flat either. As an example, we can take the curve $E/\QQ$ with label 50b4 in Cremona's tables. It admits an isogeny $\varphi$ of degree $5$ to 50b2 satisfying our hypothesis, but the reduction at $v=5$ is additive (of type II${}^*$). The type of $\varphi$ is forward. It is easy to see that the reduction of the isogeny on the connected component of $\E$ is the zero map and hence we get that the fibre of $\E[\varphi]$ over $\mathbb{F}_5$ is equal to $\mathbf{G}_a$.

\subsection{The image of the local Kummer map}

For each place $v$, there are Kummer maps 
\begin{align*}
  \kappa_v &\colon E'(K_v)/\varphi(E(K_v)) \to H^1\bigl(K_v,E[\varphi]\bigr) \cong H^1\bigl(K_v,\mu_p\bigr) \cong K_v^{\times}/p \\
  \kappa'_v & \colon E(K_v)/\hat\varphi(E'(K_v)) \to H^1\bigl(K_v,E'[\hat\varphi]\bigr) \cong H^1\bigl(K_v,\Zp\bigr) \cong \homr\bigl(K_v^{\times}/p, \Zp\bigr)
\end{align*}
The very last isomorphism is given by the local reciprocity map.
Our aim is to compare the image of $\kappa_v$ with the natural subgroup $H^1(\mathcal{O}_v,\mu_p) = \mathcal{O}_v^{\times}/p$.

We now define two finite sets of places. First
\begin{equation}\label{s1_eq}
  S_1 = \Bigl \{ \text{ all forward split multiplicative  places} \Bigr\}.
\end{equation}
In other words, it is the set of all places (above $p$ or outside $p$) where $E$ has split multiplicative reduction and the Tamagawa number of $E'$ is larger than the one of $E$. The second set is defined as
\begin{equation}\label{s2_eq}
  S_2 = \Bigl \{ \text{all backward places} \Bigr\}.
\end{equation}
So this set $S_2$ contains the places (outside $p$) where $E$ has non-split multiplicative reduction or split multiplicative reduction for which the Tamagawa number of $E$ is larger than the one of $E'$, and it contains the places above $p$ for which we have $\varphi^*(\omega') \in \mathcal{O}_v^{\times} \omega$. The two sets are disjoint.

\begin{prop}\label{kummerimage_prop}
  For all finite places $v$
  \begin{align*}
    \img(\kappa_v) &= \begin{cases}
                       K_v^{\times}/p & \text{if $v\in S_1$}\\
                       \mathcal{O}_v^{\times}/p & \text{otherwise }\\
                       0 & \text{if $v\in S_2$}  
                     \end{cases} \\
   \img(\kappa'_v) &= \begin{cases}
                      H^1(K_v,\Zp) &\text{if $v\in S_2$} \\
                      H^1(\mathcal{O}_v,\Zp) \cong \Zp & \text{otherwise}\\
                      0 & \text{if $v\in S_1$}
                     \end{cases}
  \end{align*}
\end{prop}
\begin{proof}
  Our computations in~\eqref{coker_phi_at_p_eq} and~\eqref{coker_phihat_at_p_eq} for places above $p$ and our local computations for multiplicative places, summarised in section~\ref{summary_subseq}, show that at least all of the above groups that should be equal have the same size. It remains to prove that $\img(\kappa_v)$ lies in $H^1(\mathcal{O}_v,\mu_p)=\mathcal{O}_v^{\times}/p$ and that $\img(\kappa'_v)$ lies in $H^1(\mathcal{O}_v,\Zp)$  if $v\not\in S_1\cup S_2$, i.e. if the reduction is neither forward split multiplicative nor backward.

  The image of $\kappa_v$ is equal to the image from $\E'^0(\mathcal{O}_v)/\varphi(\E^0(\mathcal{O}_v))$ if the reduction is not split multiplicative as the Tamagawa numbers are coprime to $p$. For places $v\nmid p$,  we can invoke Lemma~3.1 in~\cite{schst}, which shows that the image of $\kappa_v$ lies in the unramified part of the local cohomology, i.e. in $H^1(\mathcal{O}_v,\mathcal{E}[\varphi]) = H^1(\mathcal{O}_v,\mu_p)$. Similar for $\kappa'_v$.

  So we are reduced to the case of good or additive reduction at a place above $p$ not in $S_1$ or $S_2$. Otherwise, we can use Lemma~\ref{constant_lem} as follows. We have the corresponding Kummer map for the exact sequence
  \begin{equation*}
   \xymatrix@1{0\ar[r]& \E'^0[\hat\varphi]\ar[r] & \E'^0 \ar[r] & \E^0\ar[r] & 0 }
  \end{equation*}
  of sheaves on the fppf site, giving us
  \begin{equation*}
   \xymatrix@C-2ex{
   E(K_v)/\hat\varphi(E'(K_v)) \ar@{=}[r] & \E'^0(\mathcal{O}_v)/\hat\varphi(\E^0(\mathcal{O}_v)) \ar[d]\\
   &  H^1(\mathcal{O}_v, \E'^0[\hat\varphi]) \ar@{=}[r] & H^1(\mathcal{O}_v,\E'[\hat\varphi]) \ar@{=}[r] & H^1(\mathcal{O}_v,\Zp). }
  \end{equation*}
  This shows the equality for $\img(\kappa'_v)$. By Lemma~2.4 in~\cite{fisher}, $\img(\kappa_v)$ is the exact orthogonal to $\img(\kappa'_v)$ under the non-degenerate pairing $H^1(K_v,\mu_p)\cup H^1(K_v,\Zp) \to \Zp$. Hence, by local duality $\img(\kappa_v)$ is equal to $H^1(\mathcal{O}_v,\mu_p)$.
\end{proof}

\section{Descent via $p$-isogeny}\label{descent_sec}
There are many good and very readable accounts on descent on elliptic curves, see for instance~\cite{schst, dss, tim}. We are here in the special case of an isogeny with kernel $\mu_p$ and wish to rely fully on this assumption. However, we are only interested in the $\varphi$-Selmer group rather than the $p$-Selmer group and hence we won't perform a full descent. Note that for $p=5$ and $p=7$, Fisher's thesis~\cite{fi} contains much more for the case $K=\QQ$. In our approach, we follow and generalise~\cite{ms} for computing the Selmer groups $\Selphi$ and $\Selphihat$. However, we will reformulate their results in flat cohomology for $\mu_p$ and $\Zp$.

By local duality, see Lemma~III.1.1 and Theorem~III.1.3 in~\cite{milne_adt}, we have the short exact sequence
\begin{equation*}
 \xymatrix{0\ar[r] & H^1(\mathcal{O}_v,\mu_p) \ar[r] \ar@{=}[d] & H^1(K_v,\mu_p)\ar[r]\ar@{=}[d] &  H^1(\mathcal{O}_v,\Zp)^{\vee}\ar[r] \ar[d]^{\cong} & 0 \\
0\ar[r] & \mathcal{O}_v^{\times}/p \ar[r] & K_v^{\times}/p\ar[r]^{v} & \Zp\ar[r] & 0}
\end{equation*}
where $A^{\vee} = \homr(A,{}^{\QQ}\!/\!{}_{\ZZ})$ denotes the Pontryagin dual of $A$. The middle term is also dual to $H^1(K_v,\Zp)$.

For any open $U$ in $\Spec(\OK)$, we have an explicit description of the flat cohomology of $\mu_p$ and $\Zp$. On the one hand, we have
\begin{align*}
  H^1(U, \mu_p)  &=  \ker\Bigl( H^1 (K, \mu_p) \to \bigoplus_{v \in U} H^1(\mathcal{O}_v, \Zp)^{\vee} \Bigr ) \\
  &= \bigl\{ x \in K^{\times}/p \ \bigl\vert \ v(x) \equiv 0 \pmod{p}\  \forall \ v \in U \ \bigr\},
\end{align*}
saying that the local conditions at places in $U$ is that the image should lie in the subgroup $H^1(\mathcal{O}_v,\mu_p)$.
Furthermore this fits into the exact sequence
\begin{equation*}
  \xymatrix{ 0 \ar[r] &\mathcal{O}_{K,S}^{\times} / p \ar[r] & H^1(U,\mu_p) \ar[r] & \Cl(\mathcal{O}_{K,S})[p] \ar[r] & 0 ,}
\end{equation*}
where $S$ is the complement of $U$.

On the other hand, 
\begin{align*}
 H^1(U,\Zp) &= \ker\Bigl( H^1 (K, \Zp) \to \bigoplus_{v \in U} H^1(\mathcal{O}_v, \mu_p)^{\vee} \Bigr ).
\end{align*}
This group classifies cyclic extensions of $K$ of degree $p$ which are unramified at all places in $U$. So by global class field theory, we have
$H^1(\OK, \Zp) \cong \homr(\Cl(K), \Zp)$.

The Selmer group $\Selphi$ is defined to be the subgroup of $H^1(K, E[\varphi])$ whose restriction to $H^1(K_v,E[\varphi])$ lie in the image of the local Kummer map 
\begin{equation*}   
  \kappa_v \colon E'(K_v)/\varphi(E(K_v)) \to H^1(K_v, E[\varphi]).\end{equation*}
 Since $E[\varphi] = \mu_p$ over $K$ and over $\mathcal{O}_v$ at all places $v$ outside $p$, it is natural to compare this Selmer group to the corresponding ``Selmer group'' $H^1(\OK,\mu_p)$ for $\mu_p$. In other words, we will compare the images of the first two terms of the following two exact sequences.
\begin{equation*}
  \xymatrix@R-1ex{
   0 \ar[r] & E'(K_v)/\varphi(E(K_v)) \ar^{\kappa_v}[r] & H^1(K_v, E[\varphi]) \ar[r] \ar@{=}[d] & H^1(K_v,E)[\varphi] \ar[r] &  0 \\
   0 \ar[r] & H^1(\mathcal{O}_v,\mu_p) \ar^{\lambda_v}[r] 
            & H^1(K_v, \mu_p) \ar[r] 
            & H^1(\mathcal{O}_v,\Zp)^{\vee} \ar[r] 
            &  0 
  }
\end{equation*}
The image of $H^1(\mathcal{O}_v,\mu_p)$ can also be described as the unramified cocycles in $H^1(K_v,\mu_p)$.

Recall that we have defined in~\eqref{s1_eq} and~\eqref{s2_eq} two disjoint finite sets of places $S_1$ and $S_2$.
 For $i=1$ or $2$, let $U_i$ be the complement of $S_i$ in $\Spec(\OK)$.
 
\begin{thm}\label{pdescent_thm}
   Let $E/K$ be an elliptic curve satisfying the hypotheses~\ref{odd_asp} -- \ref{atp_hyp} in the introduction. 
   Then we have exact sequences
   \begin{equation*}
       \xymatrix@1{0\ar[r] & \Selphi \ar[r] & H^1(U_1,\mu_p) \ar[r] & \bigoplus_{v \in S_2} H^1(\mathcal{O}_v , \mu_p)}
   \end{equation*}
   and
   \begin{equation*}
       \xymatrix@1{0\ar[r] & \Selphihat \ar[r] & H^1(U_2,\Zp) \ar[r] & \bigoplus_{v \in S_1} H^1(\mathcal{O}_v , \Zp).}
   \end{equation*}
 \end{thm}

\begin{proof}
 Using Proposition~3.2 in~\cite{schst}, we immediately see that the local conditions for an element in $H^1(K,\mu_p)$ to be in $\Selphi$ are equal to the local conditions to be in $H^1(\OK,\mu_p)$ for all places where the reduction is not split multiplicative. Proposition~\ref{kummerimage_prop} tells us how to compare the images in the cases when $v$ belongs to either $S_1$ or $S_2$ or to none. We have to relax the condition at places in $S_1$ and we have to impose more stringent conditions at places in $S_2$. The same argument applies to $\Selphihat$ with the roles of $S_1$ and $S_2$ interchanged. 
\end{proof}

\subsection{Global duality}

\begin{thm}\label{global_duality_thm}
  Let $E/K$ be an elliptic curve satisfying the hypotheses~\ref{odd_asp} -- \ref{atp_hyp} in the introduction. Then we have the following exact sequence
  \begin{equation}\label{global_duality_eq}
    \xymatrix@C-2ex{ 0\ar[r] & \Selphi \ar[r] & H^1(U_1,\mu_p)\ar[r] & \bigoplus_{v \in S_2} H^1(\mathcal{O}_v, \mu_p) \ar `r[d] `[ll] `l[dlll] `d[dll] [dll] \\
                       & \Selphihat^{\vee} \ar[r] & \Cl\bigl(\mathcal{O}_{K,S_1}\bigr)/p \ar[r] & 0 }
  \end{equation}
  with the notation of Theorem~\ref{pdescent_thm}.
\end{thm}
This is of course a version of the exact sequence found by Cassels, Poitou and Tate.
\begin{proof}
 Let $U_{12}$ be the intersection of $U_1$ and $U_2$.
 Global duality in flat cohomology (see Corollary~III.3.2 and Proposition~III.0.4.c in~\cite{milne_adt}) gives that the image $A$ of $H^1(U_{12},\mu_p)$  under the localisation maps in  $\oplus_{v \in S_1\cup S_2} H^1(K_v,\mu_p)$ and the image $B$ of $H^1(U_{12},\Zp)$ in $\oplus_{v \in S_1\cup S_2} H^1(K_v,\mu_p)$ fit into the short exact sequence
 \begin{equation*}
   \xymatrix@1{0 \ar[r] & A \ar[r] & \bigoplus_{v \in S_1 \cup S_2} H^1(K_v,\mu_p) \ar[r] & B^{\vee}\ar[r] & 0.}
 \end{equation*}
 If we wish to remove the places in $S_1$ from the middle term, and replace $A$ by the image $A'$ of $H^1(U_{12},\mu_p)$ in $ \oplus_{v \in S_2} H^1(K_v,\mu_p)$, then we have to replace $B^\vee$ by $B'^{\vee}$, where $B'$ is the subgroup of $B$ of elements that are trivial in $\oplus_{v \in S_1} H^1(K_v,\Zp)$. Now we have
 \begin{equation*}
   \xymatrix@1{0 \ar[r] & A' \ar[r] & \bigoplus_{v \in S_2} H^1(K_v,\mu_p) \ar[r] & B'^{\vee}\ar[r] & 0.}
 \end{equation*}
 From Theorem~\ref{pdescent_thm}, we deduce the following two exact sequences
 \begin{equation*}
   \xymatrix@R-3ex{ 0 \ar[r] & \Selphi \ar[r] & H^1(U_{12},\mu_p)\ar[r] & \bigoplus_{v \in S_2} H^1(K_v, \mu_p), \\
              0 \ar[r] & \Selphihat \ar[r] & H^1(U_{12},\Zp) \ar[r] & \bigoplus_{v\in S_1} H^1(K_v,\Zp). }
 \end{equation*}
 This proves that $A'$ is the image of the right hand map in the top exact sequence. Also, it shows that $B'$ is the image of $\Selphihat$ in $\oplus_{v\in S_2} H^1(K_v,\Zp)$. Gathering the information  found so far, we get that
   \begin{equation*}
    \xymatrix{ 0\ar[r] & \Selphi \ar[r] & H^1(U_{12},\mu_p)\ar[r] & \bigoplus_{v \in S_2} H^1(K_v, \mu_p) \ar[r] & \Selphihat^{\vee} }
  \end{equation*}
  is exact.
  We compare now this sequence with the one we are aiming to prove:
  \begin{equation*}
    \xymatrix@C-1ex{0\ar[r] & \Selphi\ar[r] & H^1(U_{12},\mu_p)\ar[r] & \bigoplus_{v\in S_2} H^1(K_v,\mu_p) \ar[r] & \Selphihat^{\vee} \\
             0\ar[r] & \Selphi\ar[r] \ar@{=}[u] & H^1(U_1,\mu_p)\ar^(0.4){\delta}[r] \ar[u]^{\alpha} & \bigoplus_{v\in S_2} H^1(\mathcal{O}_v,\mu_p) \ar[r]\ar[u]^{\beta} &  \coker\delta \ar[u]^{\gamma}\ar[r] & 0. }
  \end{equation*}
  The injectivity of $\alpha$ and $\beta$ gives an exact sequence
  \begin{equation*}
    \xymatrix@1{0\ar[r] & \ker\gamma \ar[r] & \coker \alpha\ar[r] & \coker\beta. }
  \end{equation*}
  By definition, the cokernel of $\alpha$ is a subgroup of the group $\oplus_{v\in S_2} H^1(\mathcal{O}_v,\Zp)^{\vee}$ which is equal to the cokernel of $\beta$. Hence $\ker\gamma=0$ and we shown the exactness of the first four non-zero terms of the exact sequence~\eqref{global_duality_eq}. It remains to find the kernel of the map
  \begin{equation}\label{kerselphihat_eq}
    \xymatrix@1{\Selphihat \ar[r] &\bigoplus_{v \in S_2} H^1(\mathcal{O}_v,\mu_p)^{\vee}.}
  \end{equation}
  This is given by elements of $H^1(U_{12},\Zp)$ that map to zero in $H^1(K_v,\Zp)$ for $v\in S_1$  and to elements in $H^1(\mathcal{O}_v,\Zp)$ for $v \in S_2$. In other words it is equal to the kernel of
  \begin{equation*}
    \xymatrix@1{H^1(U_1,\Zp) \ar[r] & \bigoplus_{v\in S_1} H^1(K_v,\Zp)}
  \end{equation*}
  which, again by global duality, is dual to the kernel of $H^2(U_1,\mu_p)\to \oplus_{v\in S_1} H^2(K_v,\mu_p)$. The latter is simply the Brauer group $\Br(K_v)[p]$, while the first sits in an exact sequence
  \begin{equation*}
    \xymatrix@1{0\ar[r] & \pic(U_1)/p\ar[r] & H^2(U_1,\mu_p)\ar[r] & \Br(U_1)[p] \ar[r] & 0.}
  \end{equation*}
  By global class field theory, the localisation map on the Brauer groups is injective and so we find that the kernel of~\eqref{kerselphihat_eq} is dual to $\pic(U_1)/p=\Cl(\mathcal{O}_{K,S_1})/p$.
\end{proof}

\begin{cor}\label{global_duality_cor1}
  With the same assumption as in Theorem~\ref{global_duality_thm}, we have
  \begin{equation*}
    \dim_{\mathbb{F}_p} \Selphi - \dim_{\mathbb{F}_p} \Selphihat = \# S_1   +\#\{v \mid \infty\} -\sum_{ v\in S_2} n_v   + \dim_{\mathbb{F}_p}(\mu_p(K)) - 1 
  \end{equation*}
  where $n_v = [K_v : \QQ_p]$ if $v\mid p$ and $n_v = 1$ otherwise.
\end{cor}
\begin{proof}
  Note first that $\# H^1(U_1,\mu_p) = \#\Cl(U_1)[p] \cdot \# \mathcal{O}_{K,S_1}^{\times}/p$. Then the local term $H^1(\mathcal{O}_v,\mu_p) = \mathcal{O}_v^{\times}/p$ is of order $p$ for all places in $S_2$ outside $p$ because $\mathcal{O}_v$ contains $\mu_p$ as we discovered in section~\ref{S2_subsec}. For all places in $S_2$ above $p$, we have $\# \mathcal{O}_v^{\times}/p = p^{n_v}$  as by assumption~\ref{atp_hyp} the $p$-th roots of unity are not contained in $K_v$.
  Finally the dimension of $\mathcal{O}_{K,S_1}^{\times}/p$ is $\# S_1 +\#\{v\mid \infty\} +\dim_{\mathbb{F}_p}(\mu_p(K)) - 1$ by Dirichlet's units Theorem.
\end{proof}

Analogous to Theorem~\ref{global_duality_thm}, we could show that
  \begin{equation*}
    \xymatrix@C-2ex{ 0\ar[r] & \Selphihat \ar[r] & H^1(U_2,\Zp)\ar[r] & \bigoplus_{v \in S_1} H^1(\mathcal{O}_v, \Zp) \ar `r[d] `[ll] `l[dlll] `d[dll] [dll] \\
                       & \Selphi^{\vee} \ar[r] & H^2(U_2,\Zp)\ar[r] & 0. & } 
  \end{equation*}

\begin{rmq}
One can deduce from these computations some upper and lower bounds on the size of the groups involved. For example, with the same assumption as in Theorem~\ref{global_duality_thm}, we have, using the notations from the previous Corollary, 
  $$
\dim_{\FF_p} (\Cl\bigl(\mathcal{O}_{K,S_1}\bigr)/p)  \leq \dim_{\FF_p}\Selphihat \leq \dim_{\FF_p} (\Cl\bigl(\mathcal{O}_{K,S_1}\bigr)/p) + \sum_{v \in S_2} n_v
  $$
and a similar inequality for $\Selphi$ can be deduced via Corollary~\ref{global_duality_cor1}.
\end{rmq}

 A particular case, in which we can say more is when we have $S_2 = \emptyset$. This is the case, for example, when $p$ does not divide the product $\prod c_v$ of Tamagawa numbers of $E$ and $E$ has semistable reduction at all places above $p$.
 Write $\Sel^p(E/K)$ for the $p$-Selmer group in $H^1(K,E[p])$.

 \begin{cor}\label{global_duality_cor2}
   If $S_2=\emptyset$, then we have
   $$
   \Selphi = H^1(U_1,\mu_p)\quad\text{and}\quad \Selphihat =(\Cl\bigl(\mathcal{O}_{K,S_1}\bigr)/p)^{\vee}
   $$
   If in addition $S_1\neq\emptyset$ and $\dim_{\FF_p} (\Cl\bigl(\mathcal{O}_{K,S_1}\bigr)/p) \leq 1$, then
   $$
   \dim_{\FF_p}\Sel^p(E/K)= 2\cdot\dim_{\FF_p}(\Cl(\mathcal{O}_{K,S_1})/p) + \#S_1 +\#\{v \mid \infty\} - 2
   $$
\end{cor}

\begin{proof}
The first point follows immediately from Theorem~\ref{global_duality_thm}. Now, we have an exact sequence
$$
\xymatrix@1{
E'(K)[\hat\varphi] \ar[r] & \Selphi  \ar[r] &  \Sel^p(E/K)\ar[r] & \Selphihat \ar[r] & T  \ar[r] & 0 }
$$
where $T$ is an abelian group whose order is a square by the Cassels-Tate pairing as shown in Corollary~17 in~\cite{trihan_wuthrich} for instance.
If we assume that $\Cl(\mathcal{O}_{K,S_1})/p$ is trivial or a group of order $p$, then $T=0$. If we also assume that $S_1\neq\emptyset$, then the first map in the sequence is injective, because of Lemma~\ref{reduction_lem}. The result follows.
\end{proof}

\section{Class group pairings}\label{class_group_pairing_sec}


Recall that in this section, we do not have to make the hypotheses~\ref{odd_asp} -- \ref{atp_hyp}. Instead they will hold for any elliptic curve over a number field. To see the duality better, we formulate the first part in general for abelian varieties.

\subsection{Classical case}
\label{classpairing_subsec}

Let $A$ be an abelian variety defined over $K$, and let $\A$ be its N\'eron model over $X=\Spec(\OK)$. We denote by $\Phi=\A/\A^0 = \bigoplus_v \Phi_v$ the group of components of $\A$. Let $Z\subset X$ denotes the set of places of bad reduction of $\A$.

If $\Gamma$ is some (open) subgroup of $\Phi$, we denote by $\A^{\Gamma}$ the inverse image of $\Gamma$ by the canonical projection. Thus, $\A^{\Gamma}$ is an open subgroup scheme of $\A$ having the same generic fibre and whose component group is $\Gamma$.

Let $A^t$ be the dual abelian variety of $A$, and let $\A^t$ be its N\'eron model. Let us recall briefly that the duality between $A$ and $A^t$ is encoded in the Poincar\'e line bundle, which has a canonical structure of a biextension. We call this biextension the Weil biextension, denoted by $W_K$.

We may ask if there is a biextension of $(\A,\A^t)$ by $\mathbf{G}_{{\rm m}}$ extending $W_K$. In~\cite[expos\'e VIII, th\'eor\`eme 7.1, b)]{gro7} Grothendieck defined the so-called monodromy pairing, which is an obstruction to the existence of such a biextension.

More precisely, let $\Phi^t$ be the group of components of $\A^t$. The monodromy pairing is a bilinear map of group schemes
$$
\Phi\times\Phi^t\longrightarrow \bigoplus_{v\in Z} \QZ\cdot v
$$

In \emph{loc. cit.}, Grothendieck proved the following: If $\Gamma$ and $\Gamma'$ are subgroups that are orthogonal under the monodromy pairing, then there exists a unique biextension of $(\A^{\Gamma},\A^{t,\Gamma'})$ by $\mathbf{G}_{{\rm m}}$ extending $W_K$, which we denote by $W^{\Gamma,\Gamma'}$. The main tool we will need is not $W^{\Gamma,\Gamma'}$ itself, but rather its underlying $\mathbf{G}_{{\rm m}}$-torsor on $\A^{\Gamma}\times_S \A^{t,\Gamma'}$, which we will denote by $t(W^{\Gamma,\Gamma'})$.

We now recall the definition of the class group pairing introduced by Mazur and Tate in~\cite[Remark 3.5.3]{mt}, but going back to earlier work by Manin, N\'eron and Tate. 
Suppose again that $\Gamma$ and $\Gamma'$ are orthogonal under the monodromy pairing. The class group pairing associated to the couple $(\Gamma,\Gamma')$ is the bilinear map defined by
\begin{align*}
\A^{\Gamma}(X)\times\A^{t,\Gamma'}(X)\;\longrightarrow & \;\pic(X) \\
(Q,R)\;\longmapsto & \;\langle Q,R\rangle^{\rm cl}:=(Q\times R)^* \Bigl(t\bigl(W^{\Gamma,\Gamma'}\bigr)\Bigr).
\end{align*}

It is also possible to refine this pairing as follows. Let $S\subset X$ be a finite set of finite primes of $K$, and let $V\subseteq X$ be the complement of $S$. Then we define the $S$-class group pairing as
\begin{align*}
\A^{\Gamma}(V)\times\A^{t,\Gamma'}(V)\;\longrightarrow & \;\pic(V) \\
(Q,R)\;\longmapsto & \;\langle Q,R\rangle^{\rm cl}_S:=(Q\times R)^* \Bigl(t\bigl(W^{\Gamma,\Gamma'}\bigr)\Bigr)
\end{align*}
where $\Gamma$ and $\Gamma'$ are as before. In short, the pairing  $\langle \cdot,\cdot \rangle^{\rm cl}_{S}$ is defined on a bigger group that $\langle \cdot,\cdot\rangle^{\rm cl}$ and has values in $\Cl(\mathcal{O}_{K,S})$. If we take $S=Z$, then $\A^{\Gamma}(V)=A(K)$ so the pairing $\langle Q,R\rangle^{\rm cl}_Z$ is defined without any restriction on the points $P$ and $Q$.

\subsection{Logarithmic case}
\label{logpairing_subsec}

Let us recall here some results from~\cite{gil5}. We refer to this article for definitions and conventions concerning log schemes.

Let $U\subseteq X$ be the complement of $Z$. We denote by $X(\log Z)$ the scheme $X$ endowed with the log structure induced by $U$, see Definition~2.1.1. in~\cite{gil5}. First, we recall Proposition~4.2.1 in~\cite{gil5}.
\begin{prop}
There exists a unique biextension $W^{\rm log}$ of $(\A,\A^t)$ by $\mathbf{G}_{{\rm m}}$ for the log flat topology on $X(\log Z)$ extending Weil's biextension $W_K$.
\end{prop}
We denote by $\logpic(X,Z)$ the group $H^1_{\text{\rm kpl}}(X(\log Z), \mathbf{G}_{{\rm m}})$ of $\mathbf{G}_{{\rm m}}$-torsors for the Kummer log flat topology. From the existence of $W^{\rm log}$, we deduce a pairing
\begin{align*}
\A(X)\times\A^t(X) \;\longrightarrow & \;\logpic(X,Z) \\
(Q,R)\;\longmapsto & \;\langle Q,R\rangle^{\rm log}:=(Q\times R)^* \Bigl(t\bigl(W^{\rm log}\bigr)\Bigr)
\end{align*}
that we call the logarithmic class group pairing.

The group $\logpic(X,Z)$ has the following description (see~\cite[3.1.3]{gil5})
$$
\logpic(X,Z)= \Bigl(\Div(U)\oplus\bigoplus_{v\in Z} \mathbb{Q}\cdot v\Bigr)/\Divp(X),
$$
where $\Divp(X)$ denotes the subgroup of (usual) principal divisors on $X$.
In other words, $\logpic(X,Z)$ is the group of divisors on $X$ with rational coefficients above $Z$ modulo principal divisors. Of course, if $Z'\subseteq Z$ is a subset of $Z$, then $\logpic(X,Z')$ is a subgroup of $\logpic(X,Z)$. In particular, the group $\pic(X)$ is a subgroup of $\logpic(X,Z)$, and we have an exact sequence
\begin{equation}\label{pic_vs_log_eq}
  \xymatrix{ 0 \ar[r] & \pic(X) \ar[r] & \logpic(X,Z) \ar[r]^{\nu} & \bigoplus_{v\in Z} \QZ\cdot v  \ar[r] & 0.}
\end{equation}

The following Proposition explains how the logarithmic class group pairing is linked to the class group pairing and the monodromy pairing. We recall Proposition~4.3.2 in~\cite{gil5}.
\begin{prop}
\label{pairings_prop}
\begin{enumerate}
\item Let $\Gamma$ and $\Gamma'$ be orthogonal subgroups as before. Then the restriction of $\langle \cdot ,\cdot\rangle^{\rm log}$ to $\A^{\Gamma}(X)\times\A^{t,\Gamma'}(X)$ is equal to $\langle \cdot ,\cdot\rangle^{\rm cl}$.
\item We have a commutative diagram
$$
\xymatrix{
\A(X)\times\A^t(X)\ar[r]^{\langle \cdot ,\cdot\rangle^{\rm log}} \ar[d] & \logpic(X,Z) \ar[d]^{\nu}\\
\Phi(X)\times\Phi'(X) \ar[r] & \bigoplus_{v\in Z} \QZ\cdot v 
}$$
where the horizontal lower arrow is induced by the monodromy pairing, and the vertical arrow on the left is the reduction map.
\end{enumerate}
\end{prop}
We also have a natural surjection
$$
\logpic(X,Z) \longrightarrow \pic(U)
$$
whose kernel is the subgroup generated by classes with support in $Z$. When composing  the logarithmic class group pairing with this map, we recover the $Z$-class group pairing defined at the end of section~\ref{classpairing_subsec}.

\subsection{A remark on metrics}\label{metrics_subsec}

Let us denote by $\widehat{\pic}(\mathcal{O}_{K,Z})$ the group of locally free $\mathcal{O}_{K,Z}$-modules of rank one endowed with metrics at all places in $Z$. Then, using the description of $\widehat{\pic}(\mathcal{O}_{K,Z})$ in terms of id\`eles, it is easy to check that
\begin{equation}\label{log_vs_hat_eq}
\widehat{\pic}(\mathcal{O}_{K,Z})_{\rm tors}\simeq \logpic(X,Z).
\end{equation}
It is to note that elements in $\widehat{\pic}(\mathcal{O}_{K,Z})$ have no metrics at infinite places. If we add metrics at infinite places, then we get the so-called Arakelov class group of $\mathcal{O}_{K,Z}$. More precisely, $\widehat{\pic}(\mathcal{O}_{K,Z})$ is a quotient of the Arakelov class group by the natural subgroup isomorphic to $\mathbb{R}\times (\mathbb{R}/\mathbb{Z})^{\#\{v|\infty\}-1}$.

However, it is possible to endow the Poincar\'e line bundle on $\E_U\times \E_U$ with natural metrics at places in $Z$. As a result, we get a metrised class group pairing
$$
\E(U)\times \E(U)\rightarrow \widehat{\pic}(\mathcal{O}_{K,Z})
$$
and it is easy to check that this metrised pairing is equal to the logarithmic class group pairing via~\eqref{log_vs_hat_eq}. Thus, the logarithmic techniques we use here are closely related to the techniques of Agboola and Pappas in~\cite{ap01}. The main difference is that we do not put metrics at infinite places.

\subsection{The explicit computation of the pairing}\label{explicit_pairing_subsec}

Here we explain how we are able to compute numerically values of the logarithmic class group pairing on an elliptic curve.
Via the Proposition~\ref{pairings_prop}, this will also give an explicit formula for the class group pairing 
$$
\langle \cdot ,\cdot\rangle^{\rm cl}:\E(\OK)\times\E^0(\OK)\longrightarrow \pic(X) = \Cl(K)
$$
and we recover the monodromy pairing as described in~\cite[Example 5.8]{bosch_lorenzini}.

The first step is to give an alternative description of our pairing. Let us denote by $f\colon\E\to X=\Spec(\OK)$ the structural morphism and by $O$ the unit section of $\E$. If $Q\colon X\to\E$ is a section of $\E$, then the image of $Q$ is a divisor on $\E$, and we will denote by $[Q]$ the class of this divisor in $\pic(\E)$, just as we write $[D]$ for the class of any divisor $D$ in $\pic(\E)$ and $\logpic(\E,Z)$. If $\mathcal{L}$ is a line bundle on $\E$, we define the rigidified line bundle associated to $\mathcal{L}$ by the formula
$$
\mathcal{L}^{\rm rig}:=\mathcal{L}-f^*O^*\mathcal{L}
$$
Note that, for any point $R\in \E(\OK)$ we get $R^*\mathcal{L}^{\rm rig}=R^*\mathcal{L}-O^*\mathcal{L}$.

\begin{prop}\label{log_bundle_prop}
Let $Q\in \E(\OK)$. Then there exists a fibral divisor $F_Q$ with rational coefficients such that the element
$$
\mathcal{L}_Q:=([Q]-[O])^{\rm rig}+[F_Q]
$$
of $\logpic(\E,Z)$, satisfies the Theorem of the cube. Moreover:
\begin{enumerate}\renewcommand{\theenumi}{(\roman{enumi})}
\item\label{cl1_item} For all $R\in \E(\OK)$ we have
$$
\langle Q,R\rangle^{\rm log}=(-R)^*\mathcal{L}_Q=O^*([Q+R]-[Q]-[R]+[O])+(-R)^*[F_Q]
$$
\item\label{cl2_item} The logarithmic line bundle $(-R)^*[F_Q]$ has integral coefficients if and only if (the reductions of) $Q$ and $R$ are orthogonal under the monodromy pairing.
\item\label{cl3_item} If $R$ belongs to $\E^0(\OK)$, then $(-R)^*[F_Q]=0$.
\end{enumerate}
\end{prop}

\begin{proof}
The existence of $F_Q$, and the equality $\langle Q,R\rangle^{\rm log}=-R^*\mathcal{L}_Q$ are rather standard (see~\cite[chap. III, 1.4]{mb} and~\cite[4.3]{gil5}); note that the minus sign is coming from the fact that the generic fibre of the line bundle $\mathcal{L}_Q$ is equal to the inverse of the line bundle attached to $Q$ via the canonical isomorphism $E(K)\simeq \pic^0(E)$. Thus, we have
$$
\langle Q,R\rangle^{\rm log}=(-R)^*\mathcal{L}_Q=(-R)^*([Q]-[O])-O^*([Q]-[O])+(-R)^*[F_Q]
$$
Now, let $T_R\colon \E\to \E$ be the translation by $R$ map. For any line bundle $\mathcal{N}$ on $\E$ we have $(-R)^*\mathcal{N}=O^*T_{-R}^*\,\mathcal{N}$, and for $Q\in \E(S)$ we get $T_{-R}^*[Q]=[Q+R]$. Now, we can write
\begin{align*}
\langle Q,R\rangle^{\rm log} &=(-R)^*([Q]-[O])-O^*([Q]-[O])+(-R)^*[F_Q] \\
&=O^*([Q+R]-[R])-O^*([Q]-[O])+(-R)^*[F_Q] \\
&=O^*([Q+R]-[Q]-[R]+[O])+(-R)^*[F_Q]
\end{align*}
which proves~\ref{cl1_item}. By Proposition~\ref{pairings_prop} and the exactness of~\eqref{pic_vs_log_eq}, we see that $\langle Q,R\rangle^{\rm log}$ belongs to the usual class group if and only if $Q$ and $R$ are orthogonal to each other under Grothendieck's monodromy pairing. By the formula above, this happens if and only if $(-R)^*[F_Q]$ is a divisor with integral coefficients, so~\ref{cl2_item} is proved. Finally, if $R$ belongs to $\E^0(\OK)$, then $T_{-R}^*[F_Q]=[F_Q]$, because $F_Q$ is fibral. Moreover, $\mathcal{L}_Q$ being cubic, it is rigidified, and this immediately implies that $O^*[F_Q]=0$. Putting that together, we get
$$
(-R)^*[F_Q]=O^*T_{-R}^*[F_Q]=O^*[F_Q]=0
$$
which proves~\ref{cl3_item}.
\end{proof}

We now recall the notion of denominator ideal of a point of $E$, see~\cite{wuthfamilies}. Let $Q\in \E(\OK)=E(K)$.
Given a finite place $v$ in $\OK$, we choose a minimal Weierstrass equation for $E$ with coefficients in $\mathcal{O}_v$. 
We set 
\begin{equation*}
  e_v(Q) = -\tfrac{1}{2} \min\bigl( v(x(Q)), \ 0\bigr)  
\end{equation*}
where $x(Q)$ is the $x$-coordinate of $Q$ in our chosen equation.
On the one hand, if $Q\neq O$, then one defines  the divisor (or ideal) $e(Q)$ on $X=\Spec(\OK)$ by
$$
e(Q):=\sum_{v} e_v(Q) \cdot v.
$$
If $E$ admits a global minimal Weierstrass equation, then the ideal $x(Q)\OK$ can be written as $a(Q)\cdot e(Q)^{-2}$ for some ideals $a(Q)$ and $e(Q)$ in $\OK$ coprime to each other.

On the other hand, we set $e(O)$ to be the ideal $O^*\mathcal{O}_{\E}(O)$, where $\mathcal{O}_{\E}(O)$ is the sheaf of ideals attached to the divisor defined by $O$. Now, $e(Q)$ is well defined for any $Q\in \E(\OK)$ and, using this ideal, we can now compute the first part of $\langle Q,R\rangle^{\rm log}$.

\begin{lem}\label{explicit_pairing_lem}
Let $Q,R\in \E(\OK)$. Then
$$
\langle Q,R\rangle^{\rm log}=[e(Q+R)-e(Q)-e(R)+e(O)] + (-R)^*[F_Q]
$$
in the group $\logpic(X,Z)$. 
\end{lem}

\begin{proof}
When $Q\neq O$, one proves, using intersection theory, that the class $[e(Q)]$ is equal to the pull-back $O^*[Q]$, see~\cite[p. 250]{sil2}. Therefore, the formula follows from Proposition~\ref{log_bundle_prop}, \ref{cl1_item}. 
\end{proof}

Now we choose a Weierstrass equation with coefficients in $\OK$, corresponding to a choice of an invariant differential $\omega$ over $K$.
Then for any finite place $v$, there is an element $u_v\in \mathcal{O}_v$ such that the equation becomes minimal after a change of variables that makes $u_v^{-1} \omega$ a N\'eron differential on $\E/\mathcal{O}_v$.
\begin{lem}\label{e_zero_lem}
 If we set $e_v(O) = v(u_v)$, then we have $e(O) = \sum_v e_v(O)\cdot v\in \Cl(K)$.
\end{lem}
\begin{proof}
 From Liu's book~\cite[\S{}9.4.4, Lemma 4.29]{liu} we can extract
$$ e(O)=O^*\mathcal{O}_{\E}(O)= R^1 f_* \mathcal{O}_{\E}.$$
 By Grothendieck duality, we get 
$ e(O) \simeq (f_* \omega_{\E/X})^{-1}$ where $\omega_{\E/X}$ denotes the canonical sheaf. Now $\omega$ is a global section
of $\omega_{\E/X}$ and, at all place $v$, the N\'eron differential $u_v^{-1} \omega$ is a generator of $H^0(\mathcal{O}_v, f_*\omega_{\E/\mathcal{O}_v})$.

Finally, it is to note that the divisor class $\sum_v e_v(O) \cdot v$ does not depend on the choice of $\omega$. If we multiplied $\omega$ by an element $\alpha\in K^\times$, then the sum would get changed by adding the term $\sum_v v(\alpha^{-1}) \cdot v$ which is a principal divisor.
\end{proof}

\begin{rmq}
The values of $v(u_v)$ can easily be computed using Tate's algorithm. They also appear as correcting factors to the Tamagawa numbers in the Birch and Swinnerton-Dyer conjecture over $K$. A consequence of this Lemma, we note that the expression $e(Q)-e(O)$ behaves well under base change. 

It is a standard fact (see~\cite[\S{}9.4, Exercise 4.14]{liu}) that
$$
\Delta^{-1}\simeq (f_*\omega_{\E/X})^{12}
$$
where $\Delta$ is the minimal discriminant of $E$, which is an ideal in $\OK$. 
Hence the correction factor $\tfrac{1}{12}v(\Delta)$ in the local height formula appears as in Theorem~VI.4.1 in~\cite{sil2}.
\end{rmq}

\begin{rmq}\label{pourlescalculs_rmq}
From the Lemma, we also deduce that $e(O)$ is trivial if and only if there exists a global minimal equation over $\OK$.
 Moreover, in this case, it is possible to compute all the $e_v(Q)$ using the same coordinates.

Suppose $E/K$ is an elliptic curve which is coming by base change from an elliptic curve defined over $\mathbb{Q}$. Then, in the case that no place of additive reduction ramifies in $K/\QQ$, the global minimal equation over $\QQ$ will still be a global minimal equation over $K$.  See Proposition~V.1 and Remark~V.1.1 in~\cite{sil1}.
\end{rmq}

We now turn to the computation of the divisor $F_Q$, which will allow us to compute $\langle Q,R\rangle^{\rm log}$ completely.
Let $Q\in \E(\OK)$. By~\ref{cl3_item} of Proposition~\ref{log_bundle_prop}, the support of $F_Q$ can not contain the connected component of the fibre at any place $v$. So $F_Q$ is supported in the bad fibres and we can look individually at each type of bad reduction. So we will fix the place $v$ now and suppose that $\E/\mathcal{O}_v$ has bad reduction.

Let $\mathcal{X}/\mathcal{O}_v$ be the minimal regular model of $E/K_v$ so that $\E$ is the open subscheme of $\mathcal{X}$ obtained by removing multiple components in the special fibre.
\begin{lem}
  Let $Q\in \E(\mathcal{O}_v)$. There exists a unique fibral divisor $G_Q$ on $\mathcal{X}$ such that
  the intersection $\Gamma. \bigl( [Q] - [O] + [G_Q])$ is trivial for all fibral divisors $\Gamma$ on $\mathcal{X}$ and such that $G_Q$ intersects the connected component $\E^0$ trivially.
  Moreover, $F_Q$ is the intersection of $G_Q$ with $\E$.
\end{lem}
\begin{proof}
Recall that $F_Q$ is a fibral divisor (with rational coefficients) such that
$$
\mathcal{L}_Q=[Q]-[O]+[F_Q]
$$
satisfies the Theorem of the cube. Moreover, such an $F_Q$ is unique up to addition of an integral multiple of the special fibre of $\E$ (which is a principal divisor). Therefore, there exists a unique such $F_Q$ whose support does not contains the connected component $\Gamma_0$. Let $n$ be the exponent of the group of components of $\E$. As $nQ$ falls into the connected component of $\E$, we have
$$
\mathcal{L}_{nQ}=[nQ]-[O]
$$
that is, $[F_{nQ}]=0$. On the other hand, we have $\mathcal{L}_{nQ}=n\mathcal{L}_Q$ because the map $Q\mapsto \mathcal{L}_Q$ is a morphism of groups. Let $f$ be a meromorphic section of the generic fibre of $\mathcal{L}$, then $f^{\otimes n}$ is a meromorphic section of the generic fibre of $\mathcal{L}_{nQ}$, and this meromorphic section can be extended into a meromorphic section $\overline{f^{\otimes n}}$ of $\mathcal{L}_{nQ}$. Moreover,
$$
\mathcal{L}_Q=\Bigl[\tfrac{1}{n}\diviseur(\overline{f^{\otimes n}})\Bigr]
$$
because the right hand side is cubic and has the same generic fibre than $\mathcal{L}_Q$ (this argument is given in~\cite[Chap. 3, 1.4]{mb}).

On the other hand, we can also consider $\mathcal{L}_{nQ}=[nQ]-[O]$ as a divisor class on $\mathcal{X}$, and choose a meromorphic section $g$ extending $\overline{f^{\otimes n}}$. Let us denote by $G_Q$ the unique fibral divisor on $\mathcal{X}$ whose support does not contains $\Gamma_0$, and such that
$$
\bigl[\tfrac{1}{n}\diviseur(g)\bigr]=[Q]-[O]+[G_Q]
$$
holds in $\logpic(\mathcal{X},\{v\})$. Then the restriction of $G_Q$ to $\E$ is equal to $F_Q$. Moreover, if $\Gamma$ is an irreducible component of the special fibre of $\mathcal{X}$, then we have
$$
\Gamma.([Q]-[O]+[G_Q])=\Gamma.\bigl[\tfrac{1}{n}\diviseur(g)\bigr]=\tfrac{1}{n}\cdot \Gamma.[\diviseur(g)]
$$
and, on the other hand,
$$
\Gamma.\diviseur(g)=\Gamma. \bigl([nQ]-[O]\bigr)=0
$$
because $nQ$ and $O$ both reduce in the component $\Gamma_0$ of $\mathcal{X}$. Therefore, our $G_Q$ satisfies the intersection property. On the other hand, the uniqueness of such a $G_Q$ is proved in Silverman~\cite[Proposition 8.3]{sil2}, hence the result.
\end{proof}

 The following list gives the unique divisor $G_Q$ for each type of reduction. We fix the point $Q\in \E(\mathcal{O}_v)$ of which we can assume that it does not belong to the connected component $\E^0(\mathcal{O}_v)$ as otherwise $G_Q=F_Q =0$. Let $\E^0=\Gamma_0$, $\Gamma_1$, \dots , $\Gamma_m$ be the irreducible components of the special fibre of $\mathcal{X}$. We will determine the coefficients $a_j$ of the divisor $G_Q = a_0 \Gamma_0 + \cdots + a_m \Gamma_m$. As remarked before, we must have $a_0=0$. By the above Lemma, the intersection with an irreducible component $\Gamma$ must yield
\begin{equation}\label{GQ_eqs}
  \Gamma . G_Q = \Gamma. ([O] - [Q]) = \begin{cases} +1 &\text{ if $\Gamma = \Gamma_0$,} \\
                                         -1 &\text{ if $Q\in\Gamma$, and }\\
                                         0 &\text{ otherwise.}
                                       \end{cases}
\end{equation}
These can be expressed in terms of linear equations for the coefficients $a_i$. For all properties of the intersection pairing on $\mathcal{X}$ we refer to chapter IV in~\cite{sil2}.

\subsubsection{Multiplicative case}\label{mult_case_subsubsec}
  Suppose $E$ has split multiplicative reduction over $\mathcal{O}_v$. Say, there are $n$ connected components to $\E(\mathcal{O}_v)$; hence the Kodaira type is I$_n$. Fix a Tate parametrisation $u\colon E(K_v) \to K_v^{\times}/q^{\ZZ}$ and label the components by $\{\Gamma_j\}_{0\leq j < n}$ where $v(u(R))\equiv j\pmod{n}$ if and only if $R$ is on $\Gamma_j$. (This is well-defined up to reversing the order as $u$ can be composed with $[-1]$.) The special fibre is then  the polygon $\Gamma_0 + \Gamma_1+\cdots + \Gamma_{n-1}$.
  \begin{lem}
    Suppose $Q$  belongs to the component $\Gamma_k$ with $0 < k < n$. Then 
    \begin{equation*}
      F_Q = G_Q = \sum_{i=1}^{k} \frac{i(n-k)}{n}\Gamma_i + \sum_{i=k+1}^{n-1} \frac{(n-i)k}{n}\Gamma_i.
    \end{equation*}
   \end{lem}
  \begin{proof}
    Once the formula above is known it is a straight forward verification that the equations~\eqref{GQ_eqs} hold. We use that $\Gamma_j$ intersects $\Gamma_i$ if $ j-i \equiv \pm 1\pmod{n}$ and that the multiplicity is $1$ in this case and that $\Gamma_j.\Gamma_j = -2$ for all $j$. For instance for $j=k$, we get only contributions from the terms with $i=k-1$, $k$ and $k+1$.
    \begin{align*}
      G_Q . \Gamma_k &= \frac{(k-1)(n-k)}{n}\, \Gamma_{k-1}.\Gamma_k + \frac{k(n-k)}{n}\,\Gamma_k.\Gamma_k + \frac{(n - k-1)k}{n}\, \Gamma_{k+1}.\Gamma_k \\
      &= \frac{(k-1)(n-k)}{n}  + (-2) \frac{k(n-k)}{n} + \frac{(n-k-1)k}{n} = -1.\qedhere
    \end{align*}
  \end{proof}
    The monodromy pairing between $Q$ and $R$ in $E(K_v)$ with values in $\QZ$ is given by the degree of $-R^* [F_Q]$ by Proposition~\ref{pairings_prop}. 
    If $Q$ belongs to $\Gamma_k$ and $R$ belongs to $\Gamma_j$, then 
    this is  equal to
    \begin{equation*} -R^*[F_Q] = \begin{cases} - \frac{j(n-k)}{n} &\text{ if $j\leq k$}\\
                                                  - \frac{(n-j)k}{n} &\text{ if $j>k$}
                                    \end{cases}
    \end{equation*}
    and in both cases this gives that the monodromy pairing between $Q$ and $R$ is $jk/n$ modulo $\ZZ$.
 In the non-split multiplicative case, the situation is identical to the split multiplicative case except that $\Phi(k)$ is either trivial (when $n$ is odd) or contains only $\{\Gamma_0, \Gamma_{n/2}\}$ if $n$ is even. 

\subsubsection{Additive case}\label{additive_case_subsubsec}

\begin{description}
\item[Type II and II$^*$ :]
 $\Phi$ is trivial, so there is nothing to do.

\item[Type III :]
 $\Phi$ has order $2$ and it is obtained from two line intersecting with multiplicity $2$. So $F_Q = G_Q = \tfrac{1}{2} \Gamma_1$ will do if $Q$ belongs to $\Gamma_1$.

\item[Type III$^*$ :]
 Again $\Phi$ has order $2$, but there are now 8 components in the special fibre of $\mathcal{X}$. The connected components of $\E$ are denoted by $\Gamma_0$ and $\Gamma_1$. Then $\Gamma_0$, $\Gamma_2$, $\Gamma_4$, $\Gamma_6$, $\Gamma_5$, $\Gamma_3$, $\Gamma_1$ form a chain and $\Gamma_7$ only intersects $\Gamma_6$. Then, if $Q\in\Gamma_1$, we have
 \begin{equation*}
   G_Q = \tfrac{3}{2} \Gamma_1 + \Gamma_2  + 2\Gamma_3 + 2\Gamma_4 + \tfrac{5}{2}\Gamma_5 + 3 \Gamma_6 +\tfrac{3}{2}\Gamma_7
 \end{equation*}
and hence $F_Q = \tfrac{3}{2} \Gamma_1 $.

\item[Type IV :]
 $\Phi$ has order $3$ and it is obtained from three transversely intersecting lines. We find
 \begin{equation*}
   F_Q = G_Q = \begin{cases} \tfrac{2}{3} \Gamma_1 + \tfrac{1}{3} \Gamma_2 & \text{ if $Q \in \Gamma_1$,}\\
                         \tfrac{1}{3} \Gamma_1 + \tfrac{2}{3} \Gamma_2 & \text{ if $Q \in \Gamma_2$.}
            \end{cases}
 \end{equation*}

\item[Type IV$^*$ :]
 Still $\Phi$ has order 3, but there are 7 components. Denote by $\Gamma_6$ the unique component of multiplicity $3$. Write $\Gamma_0$, $\Gamma_1$ and $\Gamma_2$ for the elements of $\Phi$. Finally for $i=3,4,5$, let $\Gamma_i$ be the double line connecting $\Gamma_{i-3}$ and $\Gamma_6$. Then
 \begin{equation*}
   G_Q = \begin{cases}
              \tfrac{4}{3}\Gamma_1 + \tfrac{2}{3} \Gamma_2 +\Gamma_3 + \tfrac{5}{3} \Gamma_4 + \tfrac{4}{3}\Gamma_5 + 2 \Gamma_6 & \text{if $Q\in\Gamma_1$,}\\
   \tfrac{2}{3} \Gamma_1 + \tfrac{4}{3}\Gamma_2 + \Gamma_3 + \tfrac{4}{3}\Gamma_4  + \tfrac{5}{3} \Gamma_5 + 2 \Gamma_6 & \text{if $Q\in\Gamma_2$, }
            \end{cases}
 \end{equation*}
 and hence $F_Q = \tfrac{4}{3}\Gamma_1 + \tfrac{2}{3} \Gamma_2 $ or $   \tfrac{2}{3} \Gamma_1 + \tfrac{4}{3}\Gamma_2$ depending on whether $Q$ is in $\Gamma_1$ or $\Gamma_2$. Note that the contribution for IV and IV$^*$ differ by an integer.

\item[Type I$_n^*$ :]
 Label the four components of $\Phi$ by $\Gamma_0$, $\Gamma_1$, $\Gamma_2$, $\Gamma_3$. Write $\Gamma_4$ for the double line connecting $\Gamma_0$ and $\Gamma_1$. Then $\Gamma_5$, $\Gamma_6$, \dots , $\Gamma_{4+n}$ forms the chain of double line that links up to $\Gamma_2$ and $\Gamma_3$. We find
 \begin{equation*}
   G_Q = \begin{cases}
               \Gamma_1 +\tfrac{1}{2} \Gamma_2 +\tfrac{1}{2} \Gamma_3 + \Gamma_4 +\Gamma_5 +\cdots +\Gamma_{4+n} &
  \text{ if $Q \in \Gamma_1$,}\\
              \tfrac{1}{2}\Gamma_1 +\tfrac{n+4}{4} \Gamma_2 +\tfrac{n+2}{4}\Gamma_3 +\Gamma_4 +\tfrac{3}{2}\Gamma_5 +2 \Gamma_6 + \cdots + \tfrac{n+2}{2}\Gamma_{4+n} & \text{ if $Q\in\Gamma_2$,}\\
            \tfrac{1}{2}\Gamma_1 +\tfrac{n+2}{4} \Gamma_2 +\tfrac{n+4}{4}\Gamma_3 +\Gamma_4 +\tfrac{3}{2}\Gamma_5 +2 \Gamma_6 + \cdots + \tfrac{n+2}{2}\Gamma_{4+n} & \text{ if $Q\in\Gamma_3$,}\\
            \end{cases}
 \end{equation*}
 which gives that
  \begin{equation*}
   F_Q = \begin{cases}
               \Gamma_1 +\tfrac{1}{2} \Gamma_2 +\tfrac{1}{2} \Gamma_3  &
  \text{ if $Q \in \Gamma_1$,}\\
              \tfrac{1}{2}\Gamma_1 +\tfrac{n+4}{4} \Gamma_2 +\tfrac{n+2}{4}\Gamma_3  & \text{ if $Q\in\Gamma_2$,}\\
            \tfrac{1}{2}\Gamma_1 +\tfrac{n+2}{4} \Gamma_2 +\tfrac{n+4}{4}\Gamma_3  & \text{ if $Q\in\Gamma_3$.}\\
            \end{cases}
 \end{equation*}
\end{description}

\subsection{Example}

Let $E$ be the elliptic curve 158c1, which has a rational point $P=(13,-15)$ of order $5$. Let $K=\mathbb{Q}(\sqrt{-79})$, which has class group of order $h_K = 5$ generated by the ideal $\mathfrak p = (2, (1+\sqrt{-79})/2)$ above $2$. The curve $E$ has split multiplicative reduction of type I${}_{20}$ above $\mathfrak p$ and its conjugate $\bar{\mathfrak{p}}$, while the reduction over the unique place $(\sqrt{-79})$ above $79$ is non-split multiplicative of type I${}_{2}$. The set of bad places is $Z = \{\mathfrak{p},\bar{\mathfrak{p}},(\sqrt{-79})\}$ and note that the exact sequence~\eqref{pic_vs_log_eq} for $\logpic(\Spec(\OK),Z)$ does not split.
The point
$$
Q=\Bigl(\tfrac{101}{9}, - \tfrac{55}{9} + \tfrac{16}{27}\sqrt{-79} \Bigr)
$$
belongs to $E(K)$ and is of infinite order. Both points $P$ and $Q$ have good reduction at $(\sqrt{-79})$, but bad reduction over $\mathfrak{p}$ and $\bar{\mathfrak{p}}$.

While $P$ has order $5$ in the group of components $\Phi_\mathfrak{p}$, the point $Q$ has order $4$ and hence they are orthogonal under the monodromy pairing. The value of the logarithmic class group pairing must lie in the class group and we find indeed that
$$
\langle P, Q\rangle^{\rm log} = [\mathfrak{p}]^2 \in \Cl(\OK).
$$
Note however that the denominator ideals of $P$, $Q$, $O$ and 
$$P+Q = \bigl(-14 + 3\,\sqrt{-79}, -186+3\,\sqrt{-79}\bigr)$$
are all principal, so the above is the class of the integral ideal $-P^*[F_Q] = -\mathfrak{p} - 3 \bar{\mathfrak{p}}$.
Furthermore, we find
$$
\langle P,P \rangle^{\rm log} = \tfrac{4}{5} \mathfrak{p} + \tfrac{4}{5} \bar{\mathfrak{p}}\qquad\text{ and }\qquad
\langle Q,Q \rangle^{\rm log} = \tfrac{1}{4} \mathfrak{p} + \tfrac{1}{4} \bar{\mathfrak{p}}.
$$

\section{Tate-Shafarevich groups and Class groups}\label{link_sec}

\subsection{The basic connection}
Let $S\subseteq X$ be a finite set of finite primes.  We have the Kummer exact sequence for the log flat topology, see~\cite{gil5},
$$
\xymatrix@1{
0 \ar[r] & \OK^{\times}/p \ar[r] & H^1_{\text{\rm kpl}}(X(\log S),\mu_p)  \ar[r] & \logpic(X,S)[p]  \ar[r] & 0}
$$
and a canonical isomorphism
$$
H^1_{\text{\rm kpl}}(X(\log S),\mu_p)\simeq H^1(U,\mu_p)
$$
We can now use these maps to formulate the central result linking the logarithmic class group pairing to the $p$-descent. This gives a direct connection between Selmer groups (or Tate-Shafarevich groups) and class groups.

\begin{thm}\label{pairing_and_selmer_thm}
  Assume $E/K$ satisfies the hypotheses in the introduction. Let $\psi$ be the morphism defined by composing the maps
  \begin{equation*}
    \xymatrix@1{ \psi\colon\quad E'(K)/\varphi(E(K)) \ar[r]^(0.62){\kappa} & \Selphi \ar[r] & H^1(U_1,\mu_p) \ar[r]^(0.4){\rho} & \logpic(X,S_1)[p]. }
  \end{equation*}
  Then for all $Q$ in $E'(K)$, we have
  \begin{equation*}
    \psi(Q) = \langle P,Q\rangle^{\rm log}
  \end{equation*}
  where $\langle \cdot ,\cdot \rangle^{\rm log}$ is the logarithmic class group pairing.
\end{thm}

\begin{proof}
By the same arguments as in~\cite[Lemme 3.2]{gil1}, the morphism $\psi$ can be expressed in terms of the restriction of the canonical biextension $W^{\rm log}$ on $\E'\times_X \E'$. Namely, for all $Q\in \E'(X)$, we have
$$
\psi(Q)=(Q\times P)^*(t(W^{\rm log}))
$$
with notations of section~\ref{logpairing_subsec}. This is exactly the definition of the logarithmic class group pairing.
\end{proof}

\begin{rmq}
We say that the map $\psi$ defined above is the (logarithmic) class-invariant homomorphism associated to the subgroup $\E[\varphi]\subseteq \E$. This homomorphism has been introduced in a more general setting. It computes the (logarithmic) Galois module structure of $\E[\varphi]$-torsors obtained using the coboundary map $\kappa$ attached to the isogeny $\varphi$. For further discussion about $\psi$, see~\cite{gil6}.
\end{rmq}

Let us split up the map $\psi$ into
\begin{equation*}
 \xymatrix@1@C+3mm{ \psi\colon\quad E'(K)/\varphi(E(K)) \ar[r]^(0.6){\kappa} & \Selphi \ar[r]^(0.4){\psi_{\Sel}} & \logpic(X,S_1)[p] }
\end{equation*}

\begin{prop}
\label{sha_clgp_prop}
Assume $E/K$ satisfies the hypotheses in the introduction. 
 Then we have an exact sequence
$$\xymatrix@1{
0\ar[r] & \ker(\psi) \ar[r] & \ker(\psi_{\Sel}) \ar[r] & \Sha(E/K)[\varphi]  \ar[r] & \coker(\psi) \ar[r] & \coker(\psi_{\Sel}) \ar[r] & 0.}
$$
If we assume moreover that $S_2=\emptyset$, we get the exact sequence
$$\xymatrix@1{
0\ar[r] & \ker(\psi) \ar[r] &\OK^{\times}/p \ar[r] & \Sha(E/K)[\varphi]  \ar[r] & \coker(\psi) \ar[r] &  0.}
$$
\end{prop}
\begin{proof}
 The exact sequence is simply the kernel-cokernel sequence for the composition $\psi = \psi_{\Sel} \circ \kappa$. If $S_2 = \emptyset$, then $\Selphi = H^1(U_1,\mu_p)$ because the cokernel would lie in $\oplus_{v\in S_2} H^1(\mathcal{O}_v,\mu_p)$ by Theorem~\ref{pdescent_thm}. So kernel and cokernel of $\psi_{\Sel}$ identify with the ones of $\rho$.
\end{proof}

\begin{rmq}
In general, one expects that points of infinite order should not belong to the kernel of $\psi$. Let us describe more precise meanings of this statement.
\begin{enumerate}\renewcommand{\theenumi}{\arabic{enumi})}
\item \label{infpoints_item1} Starting from our point $P\in E'(K)[p]$ as usual, we believe there should exist an infinity of quadratic extensions $F/K$ together with a point $Q\in E'(F)$ of infinite order, such that $\langle P,Q\rangle^{\rm log}\neq 0$. When $K=\mathbb{Q}$ and $E'$ is semistable, we prove that in Theorem~\ref{Levinlike_thm} below (in fact, we prove a stronger statement).

\item \label{infpoints_item2} Given a point $Q$ of infinite order on $E'$, we believe there should exist a number field $F$ and a torsion point $R\in E(F)$ such that $\langle Q,R\rangle^{\rm log}\neq 0$.
\end{enumerate}
We note that~\ref{infpoints_item1} is strongly linked with the question asked by Agboola and Pappas in~\cite{ap00}. A variant of~\ref{infpoints_item2} has been explored in~\cite{at94}, namely that given $Q$ of infinite order and a fixed prime number $p$, there should exist a torsion point $R\in E(F)[p^n]$ such that $\langle Q,R\rangle^{\rm log}\neq 0$. The answer to this question is negative, as shown in~\cite{at94}.
\end{rmq}

\begin{lem}\label{reduction_lem}
If $S_1\neq\emptyset$, then $\kappa(P)\neq 0$ in the Selmer group.
\end{lem}

\begin{proof}
By construction, the point $P$ does not reduces in the connected component at places in $S_1$. Therefore, if $S_1\neq\emptyset$, the monodromy pairing being non degenerate, $\langle P,P\rangle^{\rm mono}\neq 0$. By Proposition~\ref{pairings_prop}, we deduce that $\langle P,P\rangle^{\rm log}\neq 0$. In other words, $\psi(P)\neq 0$ and, by definition, this implies that $\kappa(P)\neq 0$ in the Selmer group. This also implies that $E(K)[\varphi]=E(K)[p]$.
\end{proof}

\subsection{Splitting $\psi$ into pieces}
We denote by $\Cl(\langle S_1\rangle)$ the subgroup of $\Cl(K)$ generated by ideals with support in $S_1$.

It follows from the description of $\logpic(X,S_1)$ and the sequence~\eqref{pic_vs_log_eq} that we have an exact sequence
\begin{equation}\label{torsion_pic_vs_log_eq}
  \xymatrix@C-2ex{ 0 \ar[r] & \Cl(K)[p] \ar[r] & \logpic(X,S_1)[p] \ar[r] & \bigoplus_{v\in S_1} (\Zp) \cdot v  \ar `r[d] `[ll] `l[dlll]_{\theta} `d[dll] [dll] \\ 
    & \Cl(K)/p \ar[r] &  \Cl(\mathcal{O}_{K,S_1})/p & }
 \end{equation}
where the map $\theta$ is sending $\sum_{v\in S_1} \overline{n_v}\cdot v$ to the class of $\sum_{v\in S_1} n_v\cdot v$. We note that, if $\Cl(\langle S_1\rangle)[p]=0$  then $\theta=0$.

According to Proposition~\ref{pairings_prop}, when composing the map $\psi$ with the second map of~\eqref{torsion_pic_vs_log_eq}, we get the map
\begin{align*}
\psi^{\rm mono}:E'(K)/\varphi(E(K))\;\longrightarrow & \;\bigoplus_{v\in S_1}(\Zp)\cdot v \\
Q\;\longmapsto & \;\langle P,Q\rangle^{\rm mono}
\end{align*}
induced by the monodromy pairing.

We have an exact sequence
\begin{equation}\label{short_torsion_logpic_eq}
\xymatrix@1{
0 \ar[r] & \bigoplus_{v\in S_1}(\Zp)\cdot v \ar[r] & \logpic(X,S_1)[p]  \ar[r] & \Cl(\mathcal{O}_{K,S_1})[p] \ar[r] & 0.}
\end{equation}

When composing the map $\psi$ with the right hand side of this sequence, we get the map
\begin{align*}
\psi^{\rm cl}_{S_1}:E'(K)/\varphi(E(K))\;\longrightarrow & \;\Cl(\mathcal{O}_{K,S_1})[p] \\
Q\;\longmapsto & \;\langle P,Q\rangle^{\rm cl}_{S_1}
\end{align*}
where $\langle P,Q\rangle^{\rm cl}_{S_1}$ is the $S_1$-class group pairing defined at the end of section~\ref{classpairing_subsec}.

We note that, when composing the first map of~\eqref{torsion_pic_vs_log_eq} with the right hand side map of sequence~\eqref{short_torsion_logpic_eq}, we get the natural map $\Cl(K)[p]\rightarrow\Cl(\mathcal{O}_{K,S_1})[p]$.

\begin{rmq}\label{cokernel_comparison_rmq}
 When $\Cl(\langle S_1\rangle)[p]=0$, then $\Cl(K)[p]\simeq\Cl(\mathcal{O}_{K,S_1})[p]$, so the sequence~\eqref{torsion_pic_vs_log_eq} and~\eqref{short_torsion_logpic_eq} split, and we get an isomorphism
\begin{equation*}
\logpic(X,S)[p] \simeq \Cl(\mathcal{O}_{K,S_1})[p]\oplus \bigoplus_{v\in S_1}(\Zp)
 \cdot v .
\end{equation*}
Moreover, via this isomorphism, $\psi$ can be identified with the map $(\psi^{\rm cl}_{S_1},\psi^{\rm mono})$.
\end{rmq}

\section{Quadratic fields}\label{quadratic_sec}

\subsection{Points of infinite order and the class group pairing}
Let $S$ be a finite set of prime numbers. For any number field $K$, ee will denote by $\mathcal{O}_{K,S}$ the ring of $S'$-integers of $K$, where $S'$ is the set of finite places of $K$ lying above elements of $S$. We will denote by $\langle \cdot,\cdot\rangle^{\rm cl}_{S}$ the $S'$-class group pairing defined in section~\ref{classpairing_subsec}.

The pairing $\langle \cdot,\cdot\rangle^{\rm cl}_{S}$ has the advantage that, given two points $Q$ and $R$ on $E$, there always exists a set $S$ such that the pairing $\langle Q,R\rangle^{\rm cl}_{S}$ makes sense. For example, we may take $S$ to be the set of primes of bad reduction of $E$.

We now state a result inspired by the paper~\cite{gl}. For the convenience of the reader, we give a complete proof, following the method of~\cite{gl}. Unlike with our previous conventions, $p=2$ is allowed in this section.

\begin{thm}\label{Levinlike_thm}
 Let $E'$ be an elliptic curve defined over $\QQ$, and let $P\in E'(\mathbb{Q})$ be a point of prime order $p$.
 Let $S=\{\ell_1,\dots,\ell_r\}$ be a set of prime numbers such that $P$ reduces in the connected component of the N\'eron model of $E'$ at all primes outside $S$. Then there exists infinitely many quadratic imaginary fields $K$ with a point $Q\in E'(K)$ of infinite order, such that the class group pairing $\langle P,Q\rangle^{\rm cl}_{S}$ is non-zero in the group $\Cl(\mathcal{O}_{K,S})[p]$.
\end{thm}

\begin{proof}
Let $E$ be the quotient of $E'$ by the cyclic subgroup generated by $P$, and let $\varphi\colon E\rightarrow E'$ be the dual isogeny, that we now consider as a $\mu_p$-torsor. Because $E$ is an elliptic curve, there exists a degree $2$ morphism $\theta\colon E\to \mathbb{P}^1$ from $E$ to the projective line. Applying Hilbert's irreducibility Theorem to $\theta\circ\varphi$, we get an infinity of points $Q$ in $E'$, defined over quadratic fields, such that $\varphi^{-1}(Q)$ is a non-trivial $\mu_p$-torsor.
More precisely, there exists an infinity of quadratic imaginary fields $K$ together with a point $Q\in E'(K)$ for which the torsor $\varphi^{-1}(Q)$ is non-trivial. According to the proof of~\cite[Corollary 3.1]{gl}, we may assume that the primes in $S$ are totally ramified in these quadratic imaginary fields.

Let $K$ be such a field, and let $Q\in E'(K)$ be such a point. Of course, $\varphi^{-1}(Q)$ is just $\kappa(Q)$, where $\kappa$ is the coboundary map
$\kappa\colon E'(K)\to H^1(K,\mu_p)$
associated to the isogeny $\varphi$.
Let $U=\Spec(\mathcal{O}_{K,S})$. Then $P$ falls into the connected component of the N\'eron model of $E'$ over $U$, from which it follows, by~\cite{gil1}, that the map $\kappa$ factorises as
$$
\xymatrix@1@C+3mm{
E'(K) \ar[r]^{\kappa_U} & H^1(U,\mu_p) \ar[r] & H^1(K,\mu_p).}
$$
Let $F$ be the field such that $\kappa(Q)=\Spec(F)$, and let $L=\mathbb{Q}[\sqrt[p]{u}\,\vert\,u\in \mathbb{Z}_S^\times]$. By a variant of Hilbert's irreducibility Theorem, it is possible to ensure that $F$ is linearly disjoint from $L$. Then $K$ is linearly disjoint from $\mathbb{Q}[\zeta_p]$, which is a subfield of $L$. Combining this with the fact that the primes in $S$ are totally ramified in $K$, we have $\mathcal{O}_{K,S}^{\times}/p=\mathbb{Z}_S^{\times}/p$. Therefore, the Kummer exact sequence over $U$ reads
$$
\xymatrix@1{
0 \ar[r] & \mathbb{Z}_S^{\times}/p \ar[r] & H^1(U,\mu_p) \ar[r]^{\rho_{S}} & \Cl(\mathcal{O}_{K,S})[p] \ar[r] & 0. }
$$
Moreover, because $F$ is linearly disjoint from $L$, the torsor $\kappa_U(Q)$ does not belong to the subgroup $\mathbb{Z}_S^{\times}/p$, that is, $\rho_{S}(\kappa_U(Q))\neq 0$. On the other hand, we have the relation
$$
\rho_{S}(z)=\rho_{S}(\kappa_U(Q))=\langle P,Q\rangle^{\rm cl}_{S}
$$
which is similar to that observed in Theorem~\ref{pairing_and_selmer_thm}.
This completes the proof of the first part of the statement.

We now have to prove that $Q$ is of infinite order. Suppose $Q$ is a torsion point, then, $\kappa$ being a morphism, we may suppose that the order of $Q$ is a power of $p$. Moreover, we may ask by the same arguments as before, that $K$ is linearly disjoint from the field $\QQ\bigl(E'[p^2]\bigr)$. This implies that the $p$-part of $E'(K)$ is equal to the $p$-part of $E'(\QQ)$, in other words $Q$ belongs to $E'(\QQ)$. This is not possible because in this case the pairing $\langle P,Q\rangle^{\rm cl}_{S}$ comes by base change from a class in $\Cl(\ZZ_{S})$, and therefore is trivial. So the point $Q$ is of infinite order.
\end{proof}

\begin{exm}\label{11a2_exm}
Let $E$ be the curve labeled 11a2 in Cremona's tables and let $E'$ be 11a1. Since the same curves will appear in further examples, we describe the situation in detail here. The curve $E'$ has a rational point $P=(5,5)$ of order $5$ giving our isogeny $\hat\varphi$. The torsion point $P$ does not reduces into the connected component at the only bad prime $11$. The isogeny is forward at the split multiplicative place $11$ and the reduction is good ordinary at $5$. Therefore our hypotheses will be satisfied for any number field $K$. The set $S_2$ will be the empty set for all $K$ and $S_1$ will consist of all places above $11$. 
By Lemma~\ref{reduction_lem}, there are no non-trivial $5$-torsion points in $E(K)$.
Since $S_2 = \emptyset$,  we get $\Selphi=H^1(\mathcal{O}_{K,\{11\}},\mu_5)$.

 According to Theorem~\ref{Levinlike_thm}, there exist infinitely many imaginary quadratic fields $K$ with a point $Q\in E'(K)$ of infinite order such that $\langle P,Q\rangle^{\rm cl}_{\{11\}}$ is non-zero in the group $\Cl(\mathcal{O}_{K,\{11\}})[5]$.  Using Lemma~\ref{reduction_lem}, we know that $P$ is non-trivial in $\Selphi$, however, since $P$ is defined over $\QQ$, we have $\langle P,P\rangle^{\rm cl}_{\{11\}}=0$. 
We deduce that the sequence
$$
\xymatrix@1{
0 \ar[r] & \cyclic{5}\, P \ar[r] & \Selphi  \ar[r] & \Cl(\mathcal{O}_{K,\{11\}})[5] \ar[r] & 0 }
$$
is exact. Hence we conclude that in this case $\dim_{\FF_5}\Sha(E/K)[\varphi]$ is less than $\dim_{\FF_5}\Cl(\mathcal{O}_{K,\{11\}})[5]$. 

As a concrete example, we may consider $K=\QQ(\sqrt{-47})$ whose $\{11\}$-class group has $5$ elements. The Mordell-Weil group has rank $2$, generated by
$$ Q_1 = \bigl( 4, -\tfrac{1}{2} +\tfrac{1}{2}\sqrt{-47}\bigr) 
\qquad\text{ and }\qquad
Q_2 = \bigl( -2, -\tfrac{1}{2} +\tfrac{1}{2}\sqrt{-47}\bigr) $$
While $\psi(Q_1) = 0$ and $\langle P,Q_1\rangle^{\rm cl}_{\{11\}} = 0$, both
$\psi(Q_2)$ and $\langle P,Q_2\rangle^{\rm cl}_{\{11\}}$ are the non-trivial class containing the ideal $(7,\tfrac{11}{2} + \tfrac{1}{2}\sqrt{-47})$.
\end{exm}

\subsection{Comparison to the quadratic twists}

Let us suppose that $E$ is defined over $\QQ$ and consider $E$ over a quadratic field $K$. As usual, we can relate all the properties of $E/K$ to $E/\QQ$ and to $E^\dagger/\QQ$ where $E^\dagger$ is the quadratic twist of $E$ by the fundamental discriminant of $K$. Now, the twist admits also an isogeny $\varphi^{\dagger}\colon E^{\dagger} \to E'^{\dagger}$. 

Since $p$ is odd, all the pieces in the $p$-descent split up into a part on which $\Gal(K/\QQ)$ acts trivially and a part on which its non-trivial element acts as $-1$. We denote $()^{-}$ the latter subspaces. Since the class group of $\QQ$ is trivial, the $p$-primary part of the class group of $K$ does 
not have any non-trivial subgroup on which $\Gal(K/\QQ)$ acts trivial.

For $i=1$ and $2$, let $T_i$ be the subset of $S_i$ of places in $\QQ$ which decompose in $K/\QQ$. 
Write $U_{K,1}$ for the complement of the places above $S_1$ in $\OK$. The class group of $U_{K,1}$ is denoted by $\Cl\bigl(\mathcal{O}_{K,S_1}\bigr)$.

From~\eqref{global_duality_eq}, we deduce
the exact sequence 
\begin{equation*}
  \xymatrix{
    0\ar[r] & \Sel^{\varphi^\dagger}(E^{\dagger}/\QQ) \ar[r] & H^1(U_{K,1},\mu_p)^{-}\ar[r] & \bigoplus_{v \in T_2} H^1(\ZZ_v, \mu_p) \ar `r[d] `[ll] `l[dlll] `d[dll] [dll] \\
                       & \Sel^{\hat{\varphi}^{\dagger}}(E'^\dagger/\QQ)^{\vee} \ar[r] & \Cl\bigl(\mathcal{O}_{K,S_1}\bigr)/p \ar[r] & 0.
 }
\end{equation*}
Moreover, the second term can be computed as
\begin{equation*}
  \xymatrix@1{ 0\ar[r] & \bigoplus_{v \in T_1} \Zp \, v \oplus F \ar[r] & H^1(U_{K,1},\mu_p)^{-}\ar[r] & \Cl\bigl(\mathcal{O}_{K,S_1}\bigr)[p] \ar[r] & 0 }
\end{equation*}
where $F=0$ if $K$ is imaginary and $F=\Zp\,\eta$ if $K$ is real and $\eta$ is a fundamental unit. 

Note that the kernel of $\varphi^{\dagger}$ is not isomorphic to $\mu_p$, rather it is of the form $\mu_p\otimes\chi$ over $K$ with $\chi$ being the quadratic character corresponding to $K$. Hence the above can be viewed as an extension of our techniques of the $p$-descent to isogenies that do not satisfy our hypothesis~\ref{genmup_asp}, but do satisfy it after a quadratic extension. One could generalise this to arbitrary isogenies if one is willing to pass to a larger extension and to consider particular eigen-spaces as it is done in~\cite{ms}.

\subsection{The kernel of the class-invariant homomorphism}
\begin{lem}
  Suppose $K$ is a imaginary quadratic field with $\mu_p\not\subset K$. Then $\psi_{\Sel}$ and $\psi$ are injective and
  \begin{equation*}
    \xymatrix@1{0\ar[r]& \Sha(E/K)[\varphi]\ar[r] & \coker(\psi)\ar[r] & \coker(\psi_{\Sel})\ar[r] & 0 }
  \end{equation*}
 is exact.
\end{lem}
\begin{proof}
  The map $\psi_{\Sel}$ factors as 
  $$\xymatrix@1@C+3mm{
  \psi_{\Sel}\colon \Selphi \ar[r] & H^1(U_1,\mu_p)\ar[r]^(0.4){\rho} & \logpic(X,S_1)[p]
   }$$
  The first map is injective. The kernel of the second map $\rho$ is $\OK^{\times}/p$, which is trivial if $K$ is an imaginary quadratic field unless $p=3$ and $K=\QQ(\mu_3)$. The rest follows from Proposition~\ref{sha_clgp_prop}.
\end{proof}

If we are in the case the $S_2$ is empty, then we deduce further an isomorphism $\Sha(E/K)[\varphi] = \coker(\psi)$. Furthermore, the map $\psi_{\Sel}$ is an isomorphism from $\Selphi$ to $\logpic(X,S_1)[p]$. Hence one could proceed to do the $p$-descent in the  logarithmic class group avoiding completely the explicit Kummer map $\kappa$.

If instead $K$ is real quadratic, the kernel may be non-trivial, but it will have at most $p$ elements. In fact, it is the intersection of 
$\OK^{\times}/p\cong \Zp$ and $\Selphi$ inside $H^1(U_1,\mu_p)$. Two examples are given in section~\ref{real_ex_subsec} below.

\section{Numerical Examples}\label{numerical_sec}

All numerical examples in this paper were computed using Sage~\cite{sage}. Underlying to it, PARI/GP~\cite{parigp} was heavily used. Furthermore some calculations involved the use of Magma~\cite{magma}.
The algorithms, implemented to be used in Sage, for computing the Selmer groups and the logarithmic class group pairing are available on the authors web pages.

\subsection{An imaginary quadratic example}
Let us consider the example~\ref{11a2_exm} of the curves $E=$ 11a2 and $E'$ =11a1 and the point $P=(5,5)$ over an imaginary quadratic field $K$. There are now two cases. Either the order of the places above $11$ in the class group is divisible by $p=5$ or not, in other words either $\theta$ in sequence~\eqref{torsion_pic_vs_log_eq} is trivial or has an image of order $5$. Write $t\in\{0,1\}$ for the dimension of the image of $\theta$ and write $h_5$ for the dimension of $\Cl(\OK)[5]$ as $\FF_5$-vector spaces. In particular if $11$ is inert or ramified in $K$, then $t=0$.
We find that 
\begin{align*}
 \dim_{\FF_5} \Selphi &= \dim_{\FF_5} \logpic(X,S_1)[5] = h_5 + \# S_1 - t \text{ and }\\
 \dim_{\FF_5}\Selphihat &= h_5.
\end{align*}
If we assume that $h_5 \leq 1$, then, by Corollary~\ref{global_duality_cor2}, we get
$$
\dim_{\FF_5} \Sel^5(E/K) =2 h_5 + \#S_1 - t - 1.
$$

\subsection{Real quadratic examples}\label{real_ex_subsec}
In the case, when $K$ is a real quadratic field, the map $\psi$ may or may not be injective.
First we explain a case, when $\psi$ is {\emph not} injective. Let $K=\QQ(\sqrt{2})$ and take again the isogeny $\varphi$ from 11a2 to 11a1 of example~\ref{11a2_exm}.
The class group of $K$ is trivial. The Mordell-Weil group $E'(K)$ is generated by our $5$-torsion point $P$ and the point $Q = \bigl(\tfrac{9}{2}, -\tfrac{1}{2} +\tfrac{7}{4} \sqrt{2}\bigr)$ of infinite order. Since $S_1$ is not empty $\psi(P)\neq 0$ by Lemma~\ref{reduction_lem}. Since $Q$ has good reduction at $(11)$, we have $\psi(Q) = \langle P,Q\rangle^{\rm log} = 0$. Note that $Q$ is not in the image of $\varphi$ as $E(K) = \ZZ\, \hat\varphi(Q)$. It follows that the kernel of $\psi\colon  E'(K)/\varphi(E(K)) \to \logpic(X,S_1)[p]$ is generated by the image of $Q$. In other words, the image of $Q$ under
$E'(K)/\varphi(E(K)) \to \Selphi \to H^1(U_1,\mu_p)$ lands in the image of the global units $\OK^{\times}/p$. 

If we consider instead the isogeny $\varphi$ from 35a2 to 35a1 of degree $3$, still over $K=\QQ(\sqrt{2})$, we find that $\psi$ is injective. The rank of $E'(K)$ is again $1$, generated by a point $Q = \bigl(\tfrac{9}{2},-\tfrac{1}{2} +\tfrac{35}{4}\sqrt{2}\bigr)$. But this time
$$\psi(Q) = \frac{2}{3} \cdot (1+2\sqrt{2}) + \frac{1}{3} \cdot (1-2\sqrt{2}).$$ 

\subsection{A cubic example}

We consider our example~\ref{11a2_exm} of the curve 11a2 and 11a1 over the cubic field
$K=\QQ(\mu_7)^{+}$ obtained by adjoining the root $\alpha$ of $x^3+x^2 -2\,x-1=0$. 

The class group of $K$ is trivial and $E(K)$ does not contain any torsion points. The group $\Selphi = H^1(U_1,\mu_5)$ has the explicit basis $11$, $\alpha+1$, $\alpha^2-1$ as a $\FF_5$-vector space in $K^\times/p$. We have $\Selphihat = 0$. We deduce that the dimension of $\Sel^5(E/K)$ is $2$ and that the dimension of $\Sel^5(E'/K)$ is at most $3$ because it is a subspace of $\Selphi$.

As $K/\QQ$ is abelian and $E$ is defined over $\QQ$, we can apply Kato's Theorem~\cite{kato}. Firstly, the Tate-Shafarevich groups are finite and, secondly, the computation of the algebraic value of the $L$-series of $E$ over $K$ using modular symbols reveals that the rank of $E(K)$ is $0$. The Birch and Swinnerton-Dyer conjecture asserts that $\Sha(E/K)$ has $5^4$ elements and $\Sha(E'/K)$ has $5^2$. Kato's result and a $2$-descent also show that no prime other than $5$ can divide the order of these Tate-Shafarevich groups.

From our 5-descent information, we conclude that $\Sha(E/K)[5] = (\cyclic{5})^2$. Doing a further $5$-descent with the isogeny from $E'=$  11a1 to $E''=$ 11a3 (whose kernel contains a point of order $5$), we find that $\Sha(E'/K)[5] = (\cyclic{5})^2$, too. However, our descents will not reveal if $\Sha(E''/K)[5]$ is trivial (as BSD predicts) or not. Using the invariance of the BSD formula under isogenies, we see that $\Sha(E/K)$ will either have $5^4$ or $5^6$ elements.

 We conclude that
\begin{alignat*}{2}
  E(K) &= 0\qquad & \qquad E'(K) &= \cyclic{5} P\\
  \Sel^5(E/K) &= \bigl(\cyclic{5}\bigr)^2 \qquad & \qquad \Sel^5(E'/K) &= \bigl(\cyclic{5}\bigr)^3 \\
  \Sha(E/K) &= \bigl(\cyclic{25}\bigr)^2\text{ or }\bigl(\cyclic{125}\bigr)^2 \qquad & \qquad\Sha(E'/K)&= \bigl(\cyclic{5}\bigr)^2\text{ or }\bigl(\cyclic{25}\bigr)^2 .
\end{alignat*}
 
\subsection{A more complicated cubic example}
We take the curve $E=$ 294b1 given by the equation
$$ y^2 \ +\ x\,y \ =\ x^3\ -\ x\ -\ 1 $$
which admits an isogeny $\varphi$ of degree $7$ to the curve $E'=$ 294b2. The kernel of the dual is generated by a point $P=(6,3)$ on $E'$.
We consider this curve over the cubic extension generated by $\alpha$ which satisfies $\alpha^3 + \alpha^2 - 104\,\alpha + 371 = 0$. This is an abelian extension of conductor $313$. Its class number is $7$.

The set $S_1$ consists of the two forward split multiplicative places $(2)$ and $(3)$. The set $S_2$ instead contains the three places above $7$, namely $(7,\alpha)$, $(7,\alpha+3)$ and, $(7,\alpha-2)$. As an explicit basis for $H^1(U_1,\mu_7)$, we can take $\{2, 3, \alpha - 5, \alpha - 6, 8\alpha^2 + 53\alpha - 513\}$ where the last represents the generator of the seven power of a non-principal ideal.
The localisation map $H^1(U_1,\mu_7) \to \oplus_{v\in S_2} H^1(\mathcal{O}_v,\mu_7)$ is a linear map from a vector space of dimension $5$ to a space of dimension $3$. It can be checked that this map is not surjective, rather its cokernel has dimension $1$. We conclude that 
$$ \dim_{\FF_7} \Selphi = 3 \qquad \text{ and }\qquad \dim_{\FF_7} \Selphihat = 2. $$

With some luck, we are able to find a point of infinite order on $E'(K)$, namely $Q = (2\,\alpha^2 - 25\,\alpha + 69, 50\,\alpha^2 - 625\,\alpha + 1921)$ and a point $R= (\alpha^2 - 10\,\alpha + 29, 17\,\alpha^2 - 170\,\alpha + 415)$ on $E(K)$ of infinite order. It is possible to show using the N\'eron-Tate height pairing, that $P$, $\sigma(P)$, $\varphi(R)$, and $\sigma(\varphi(R))$ are linearly independent in $E'(K)$. From our descent computations, we conclude that $\Sha(E/K)[7] = \Sha(E'/K)[7] = 0$ and that $E'(K)/\varphi(E(K)) = \Selphi$ has dimension 3, while $E(K)/\hat\varphi(E'(K))=\Selphihat$ has dimension $2$. It is not difficult to see that the subgroup of points we found has index coprime to $7$ in  the full Mordell-Weil group.  In fact, the Birch and Swinnerton-Dyer conjecture asserts that the index is equal to the order of $\Sha(E'/K)$. 
We get that $P$, $Q$ and $\sigma(Q)$ generate $E'(K)/\varphi(E(K))$ while $R$ and $\sigma(R)$ generate $E(K)$ modulo the image of $\hat\varphi$.

We now analyse the class-invariant homomorphism 
$$\psi\colon E'(K)/\varphi(E(K)) \to \logpic(X,S_1)[7].$$
 The source is a $\FF_7$-vector space of dimension $3$ generated by $P$, $Q$ and $\sigma(Q)$. Since the only two places in $S_1 = \{(2), (3)\}$ are principal ideals, we get
$$
\xymatrix@1{ 0 \ar[r] & \Cl(K)[7]\ar[r] & \logpic(X,S_1)[7] \ar[r] & \cyclic{7}\,(2) \oplus \cyclic{7}\, (3) \ar[r] & 0. }
$$
Hence the target is also a three dimensional space, generated by $\tfrac{1}{7}(2)$, $\tfrac{1}{7}(3)$, and $\mathfrak{c} = (7,\alpha)$. It is easy to compute the images of the basis by computing the logarithmic class group pairing. We have
\begin{align*}
  \psi(P) &= \langle P, P \rangle^{\rm log} = \tfrac{1}{7} (3) + \tfrac{4}{7} (2), \\
  \psi(Q) &= \langle P , Q\rangle^{\rm log} = \mathfrak{c}^2,\\ 
  \psi(\sigma(Q)) &= \langle P,\sigma(Q) \rangle^{\rm log}  = \mathfrak{c}^4
\end{align*}
and hence $Q+3\,\sigma(Q)$ is in the kernel of $\psi$. In fact both kernel and cokernel of $\psi$ are of dimension $1$.

\subsection{A quintic example}
We return to our isogeny $\varphi\colon E \to  E'$ with $E=$ 11a2 and $E'= $ 11a1 from example~\ref{11a2_exm}. But now over the abelian quintic field  $K$ of conductor $61$. A simple element $\alpha$ for $K$ is given by a root of $x^5 + x^4 - 24\,x^3 - 17\,x^2 + 41\,x - 13 = 0$. Its class group is trivial. We deduce that the dimension of $\Selphi$ is $9$ and $\Selphihat$ is trivial. It follows from the fact that $E(K)$ is finite (again shown using modular symbols) that $\Sha(E/K)[5] = \Sha(E/K)[\varphi]$ is of dimension $8$. The Birch and Swinnerton-Dyer conjecture asserts that the order of $\Sha(E/K)$ is $5^8\cdot 19^2$. 

\section*{Acknowledgements}
The authors would like to thank Adebisi Agboola and Sylvia Guibert for interesting and helpful discussions.

\bibliographystyle{amsplain}
\bibliography{sha}

\providecommand{\bysame}{\leavevmode\hbox to3em{\hrulefill}\thinspace}
\providecommand{\MR}{\relax\ifhmode\unskip\space\fi MR }
\providecommand{\MRhref}[2]{%
  \href{http://www.ams.org/mathscinet-getitem?mr=#1}{#2}
}
\providecommand{\href}[2]{#2}
\begin{thebibliography}{10}

\bibitem{a94}
Adebisi Agboola, \emph{A geometric description of the class invariant
  homomorphism}, J. Th\'eor. Nombres Bordeaux \textbf{6} (1994), no.~2,
  273--280.

\bibitem{ap00}
Adebisi Agboola and George Pappas, \emph{Line bundles, rational points and
  ideal classes}, Math. Res. Letters \textbf{7} (2000), 1--9.

\bibitem{ap01}
\bysame, \emph{On arithmetic class invariants}, Math. Annalen \textbf{320}
  (2001), 339--365.

\bibitem{at94}
Adebisi Agboola and Martin~J. Taylor, \emph{Class invariants of
  {M}ordell-{W}eil groups}, J. Reine Angew. Math. \textbf{447} (1994), 23--61.

\bibitem{bosch_lorenzini}
Siegfried Bosch and Dino Lorenzini, \emph{Grothendieck's pairing on component
  groups of {J}acobians}, Invent. Math. \textbf{148} (2002), no.~2, 353--396.

\bibitem{magma}
Wieb Bosma, John Cannon, and Catherine Playoust, \emph{The {M}agma algebra
  system. {I}. {T}he user language}, J. Symbolic Comput. \textbf{24} (1997),
  no.~3-4, 235--265, Computational algebra and number theory (London, 1993).

\bibitem{ct95}
Philippe Cassou-Nogu\`es and Martin~J. Taylor, \emph{Structures galoisiennes et
  courbes elliptiques}, J. Th\'eor. Nombres Bordeaux \textbf{7} (1995), no.~1,
  307--331.

\bibitem{cm00}
John Cremona and Barry Mazur, \emph{Visualizing elements in the
  {S}hafarevich-{T}ate group}, Experimental Math. \textbf{9} (2000), 13--28.

\bibitem{dss}
Z.~Djabri, Edward~F. Schaefer, and N.~P. Smart, \emph{Computing the
  {$p$}-{S}elmer group of an elliptic curve}, Trans. Amer. Math. Soc.
  \textbf{352} (2000), no.~12, 5583--5597.

\bibitem{tim}
Tim Dokchitser, \emph{Deformations of {$p$-divisible} groups and {$p$-descent}
  on elliptic curves}, Ph.D. thesis, Proefschrift Universiteit Utrecht, 2000.

\bibitem{fi}
Tom Fisher, \emph{On 5 and 7 descents for elliptic curves}, Ph.D. thesis,
  Cambridge University, UK, 2000.

\bibitem{fisher}
\bysame, \emph{The {C}assels-{T}ate pairing and the {P}latonic solids}, J.
  Number Theory \textbf{98} (2003), no.~1, 105--155.

\bibitem{gil1}
Jean Gillibert, \emph{Invariants de classes : le cas semi-stable}, Compos.
  Math. \textbf{141} (2005), no.~4, 887--901.

\bibitem{gil5}
\bysame, \emph{Prolongement de biextensions et accouplements en cohomologie log
  plate}, Int. Math. Res. Not. IMRN (2009), no.~18, 3417--3444.

\bibitem{gil6}
\bysame, \emph{Cohomologie log plate, actions mod\'er\'ees et structures
  galoisiennes}, J. Reine Angew. Math. (2011), no.~??, ??

\bibitem{gl}
Jean Gillibert and Aaron Levin, \emph{Pulling back torsion line bundles to
  ideal classes}, Preprint, 2011.

\bibitem{gro7}
Alexandre Grothendieck, \emph{Groupes de monodromie en g\'eom\'etrie
  alg\'ebrique. {I}}, Lecture Notes in Mathematics, Vol. 288, Springer-Verlag,
  Berlin, 1972, S{\'e}minaire de G{\'e}om{\'e}trie Alg{\'e}brique du Bois-Marie
  1967--1969 (SGA 7 I), Dirig{\'e} par A. Grothendieck. Avec la collaboration
  de M. Raynaud et D. S. Rim.

\bibitem{kato}
Kazuya Kato, \emph{{$p$}-adic {Hodge} theory and values of zeta functions of
  modular forms}, Cohomologies {$p$}-adiques et application arithm\'etiques.
  {III}, Ast\'erisque, vol. 295, Soci\'et\'e Math\'ematique de France, Paris,
  2004.

\bibitem{liu}
Qing Liu, \emph{Algebraic geometry and arithmetic curves}, Oxford Graduate
  Texts in Mathematics, vol.~6, Oxford University Press, Oxford, 2002,
  Translated from the French by Reinie Ern{\'e}, Oxford Science Publications.

\bibitem{mt}
Barry Mazur and John Tate, \emph{Canonical height pairings via biextensions},
  Arithmetic and geometry, {V}ol. {I}, Progr. Math., vol.~35, Birkh\"auser
  Boston, Boston, MA, 1983, pp.~195--237.

\bibitem{ms}
Robert Miller and Michael Stoll, \emph{Explicit isogeny descent on elliptic
  curves}, \url{http://arxiv.org/abs/1010.3334}, 2010.

\bibitem{milne_adt}
James~S. Milne, \emph{Arithmetic duality theorems}, second ed., BookSurge, LLC,
  Charleston, SC, 2006.

\bibitem{mb}
Laurent Moret-Bailly, \emph{Pinceaux de vari\'et\'es ab\'eliennes},
  Ast\'erisque (1985), no.~129, 266.

\bibitem{pa98}
George Pappas, \emph{On torsion line bundles and torsion points on abelian
  varieties}, Duke Math. J. \textbf{91} (1998), no.~2, 215--224.

\bibitem{sch}
Edward~F. Schaefer, \emph{Class groups and {S}elmer groups}, J. Number Theory
  \textbf{56} (1996), no.~1, 79--114.

\bibitem{schst}
Edward~F. Schaefer and Michael Stoll, \emph{How to do a {$p$}-descent on an
  elliptic curve}, Trans. Amer. Math. Soc. \textbf{356} (2004), no.~3,
  1209--1231.

\bibitem{sil2}
Joseph~H. Silverman, \emph{Advanced topics in the arithmetic of elliptic
  curves}, Graduate Texts in Mathematics, vol. 151, Springer-Verlag, New York,
  1994.

\bibitem{sil1}
\bysame, \emph{The arithmetic of elliptic curves}, second ed., Graduate Texts
  in Mathematics, vol. 106, Springer, Dordrecht, 2009.

\bibitem{sage}
W.\thinspace{}A. Stein et~al., \emph{{S}age {M}athematics {S}oftware ({V}ersion
  4.7.1)}, The Sage Development Team, 2011, available from
  \url{http://www.sagemath.org}.

\bibitem{t88}
Martin~J. Taylor, \emph{Mordell-{W}eil groups and the {G}alois module structure
  of rings of integers}, Illinois J. Math. \textbf{32} (1988), no.~3, 428--452.

\bibitem{parigp}
{The PARI~Group}, Bordeaux, \emph{{PARI/GP, version {\tt 2.4.3}}}, 2011,
  available from \url{http://pari.math.u-bordeaux.fr/}.

\bibitem{trihan_wuthrich}
Fabien Trihan and Christian Wuthrich, \emph{Parity conjectures for elliptic
  curves over global fields of positive charactersitic}, Compositio Math.
  \textbf{147} (2011), no.~4, 1105--1128.

\bibitem{wuthfamilies}
Christian Wuthrich, \emph{On {$p$}-adic heights in families of elliptic
  curves}, J. London Math. Soc. (2) \textbf{70} (2004), no.~1, 23--40.

\end{thebibliography}

\end{document}